\documentclass[11pt, leqno]{amsart}

\usepackage{amssymb}
\usepackage{amsmath,amscd}
\usepackage[initials,nobysame,alphabetic]{amsrefs}
\usepackage{amsmath, amssymb}
\usepackage{amsfonts}
\usepackage{mathrsfs}
\usepackage{mathpazo}
\usepackage{mathtools}
\usepackage[arrow,matrix,curve,cmtip,ps]{xy}
\usepackage[utf8x]{inputenc}
\usepackage{amsthm}

\usepackage{float}
\usepackage{hyperref}
\usepackage{graphics}
\usepackage{graphicx}
\usepackage{verbatim}
\usepackage{multirow}
\usepackage{tikz}
\usepackage{enumerate}
\allowdisplaybreaks

\newtheorem{theorem}{Theorem}[section]
\newtheorem{lemma}[theorem]{Lemma}
\newtheorem{proposition}[theorem]{Proposition}
\newtheorem{corollary}[theorem]{Corollary}

\newtheorem*{theorem*}{Theorem}

\newtheorem{definition}[theorem]{Definition}

\def\as{\hbox{\rm a.s.{ }}}

\numberwithin{equation}{section}

\usepackage{color}
\newtheorem{Assumption}{Assumption}[part]

\def \g{\gamma}
\def \t{\tau}
\def \ind{1\!\!1}

\newcommand{\nc}{\newcommand}
\nc{\esssup}{\mathop{\mathrm{ess\,sup}}}
\nc{\essinf}{\mathop{\mathrm{ess\,inf}}}
\nc{\argmax}{\mathop{\mathrm{arg\,max}}}

\def \ms {\medskip}

\def \P{\mathbb{P}}

\def \R{\mathbb{R}}

\def \E{\mathbb{E}}
\def \F{\mathbb{F}}

\def \Q{\mathbb{Q}}

\def \1{\mathbf{1}}

\def \Fc{\mathcal{F}}

\def \Sc{\mathcal{S}}
\def \Tc{\mathcal{T}}

\def \eps{\varepsilon}

\def \ed {\end{document}}
\def \a {\alpha}

\def\enqs{\end{eqnarray*}}
\def\beq{\begin{eqnarray}}
\def\enq{\end{eqnarray}}

\newcommand{\Lim}     {\displaystyle \lim\limits}

\def \ga {\gamma_1}
\def \gb {\gamma_2}

\title{Time-inconsistent mean-field optimal stopping: A limit approach}
\author{ Boualem Djehiche and Mattia Martini }
\address{Department of Mathematics \\ KTH Royal Institute of Technology \\ 100 44, Stockholm \\ Sweden}
\email{boualem@kth.se}
\address{Dipartimento di Matematica ”Federigo Enriques”, Universit{\`a} degli Studi di
Milano, Via Saldini 50, 20133 Milano, Italy}
\email{mattia.martini@unimi.it}

\date{This version: March 17, 2023}



\begin{document}

\begin{abstract}
    We provide a characterization of an optimal stopping time for a class of finite horizon time-inconsistent optimal stopping problems (OSPs) of mean-field type, adapted to the Brownian filtration, including those related to mean-field diffusion processes and recursive utility functions.
    Despite the time-inconsistency of the OSP, we show that it is optimal to stop when the value-process hits the reward process for the first time, as is the case for the standard time-consistent OSP. We solve the problem by approximating the corresponding value-process with a sequence of Snell envelopes of processes, for which a sequence of optimal stopping times is constituted of the hitting times of each of the reward processes by the associated value-process. Then, under mild assumptions, we show that this sequence of hitting times converges in probability to the hitting time for the mean-field OSP and that the limit is optimal. 
    \end{abstract}

\maketitle
\tableofcontents

\section{Introduction}
Optimal stopping problems (OSP) with cost depending on the mean of the stopped process arise for instance when the goal is to minimize the variance. In the work by Pedersen and Peskir \cites{pedersen11, pedersenpeskir16}), the problems of optimal variance stopping and optimal mean-variance stopping  have been investigated in the case where the underlying process is e.g. a Geometric Brownian motion, highlighting connections with portfolio choice. Despite the fact that the OSP of the variance is {\it time-inconsistent} i.e. for which the value-process does not satisfy the Bellman equation, they succeeded to solve the problem i.e. derive a variational inequality for the value function with an explicit stopping region.  
More recently, the interest in optimal stopping problems with a more general mean-field type interaction such as the dependence on the law of the stopped process increased significantly, due mainly to the connection with the theory of mean-field games and mean-field optimal control (see e.g. \cites{carmonadelarue1} for a systematic presentation of the topic). The first contribution in this direction is the work by Bertucci \cite{bertucci} where an optimal stopping problem for a mean-field game is studied using mainly PDE techniques. Then, other results and extensions have been obtained with different techniques in \cites{carmona2017mean, nutz2018mean, bouveretdumitrescutankov20, dumitresculeutschertankov21} and in the recent papers \cites{talbitouzizhang21, dumitresculeutschertankov22}. 

A powerful tool to study optimal stopping problems is based on the Snell envelope of processes (see \cites{bismut1977, elkaroui79} and Appendix D in \cite{karatzasshreve98}). Given the general OSP 
\begin{equation*}
	Y_0=\underset{\tau\in \mathcal{T}_0}{\sup}\, \E\left[L_\tau\right]
\end{equation*}
where the reward (a.k.a. the barrier or obstacle) process $(L_t)_{t\in[0,T]}$ is right continuous with left limits (c{\`a}dl{\`a}g) and adapted to a filtration $\mathbb{F}=\{\mathcal{F}_t\}_{t\in[0,T]}$ which satisfies the usual conditions, the Snell envelope of the process $L$ is defined by
\begin{equation*}
	Y_t=\underset{\tau\in \mathcal{T}_t}{\esssup}\, \E\left[L_\tau\,|\,\Fc_t\right],
\end{equation*}
where $\mathcal{T}_t$ is set of $\mathbb{F}$-stopping times with values in $[t,T]$. Under mild uniform integrability conditions on $L$, it turns out that $Y$ satisfies the Bellman equation: for any $\sigma\in\Tc_0$ and $\tau\in\Tc_{\sigma}$, we have
\begin{equation}\label{Bellman}
\E[Y_{\tau}\,\lvert\, \Fc_{\sigma}]=\underset{\rho\in\mathcal{T}_{\tau}}{\esssup}\,\E[Y_{\rho}\,\lvert\, \Fc_{\sigma}] \quad \as,
\end{equation}  since, for this type of obstacles, the conditional expectation is closed under pairwise maximization (see e.g Lemma D.1 in \cite{karatzasshreve98}) i.e. for any $\rho_1,\rho_2$  in $\Tc_{\tau}$, it holds that
\begin{equation}\label{pair-max}
\E[Y_{\rho_3}\,\lvert\, \Fc_{\sigma}]=\ind_{A}\E[Y_{\rho_1}\,\lvert\, \Fc_{\sigma}]+\ind_{A^c}\E[Y_{\rho_2}\,\lvert\, \Fc_{\sigma}]=\E[Y_{\rho_1}\,\lvert\, \Fc_{\sigma}]\vee \E[Y_{\rho_2}\,\lvert\, \Fc_{\sigma}],
\end{equation}
where
$$
A:=\{\E[Y_{\rho_1}\,\lvert\, \Fc_{\sigma}]\ge \E[Y_{\rho_2}\,\lvert\, \Fc_{\sigma}]\}, \quad \rho_3=\rho_1\ind_{A}+ \rho_2\ind_{A^c}.
$$

The Bellman equation \eqref{Bellman} implies that the value process is the smallest supermartingale that dominates $L$.
This important property implies that, when $L$ has only nonnegative jumps, $Y$ is continuous and it is optimal to stop when $Y$ hits $L$ i.e. the hitting time
\begin{equation}\label{tau-intro}
	\tau^*_t = \inf\{s\ge t, \,\, Y_s=L_s\}\wedge T
\end{equation}
is optimal after $t$. In particular, $\tau^*:=\tau^*_0$
is optimal for $Y_0$.

If the obstacle process $L$ is of mean-field type such as being of the form $L_t=h(X_t,\E[X_t])$, the associated value process 
\begin{equation*}
	Y_t=\underset{\tau\in\mathcal{T}_t}{\esssup}\, \E\left[h(X_{\tau},\E[X_{\tau}])\,|\,\Fc_t\right]
\end{equation*}
satisfies neither \eqref{pair-max} nor the Bellman equation \eqref{Bellman}. But, if instead of the expected value of the random variable $X_{\tau}$, we consider the 'deterministic' function $\phi(s):=\E[X_s]$ evaluated at $s=\tau$, the value process 
\begin{equation*}
	Y_t:=\underset{\tau\in\mathcal{T}_t}{\esssup}\, \E\left[h(X_{\tau},\E[X_{s}]\lvert_{s=\tau}))\,|\,\Fc_t\right]=\underset{\tau\in\mathcal{T}_t}{\esssup}\, \E\left[h(X_{\tau},\phi(\tau)))\,|\,\Fc_t\right]
\end{equation*}
does satisfy \eqref{pair-max} which yields the Bellman equation \eqref{Bellman} and thus  the stopping time
\begin{equation}\label{tau-intro-mf}
	\tau^*_t = \inf\{s\ge t, \,\, Y_s=h(X_s, \phi(s))\}\wedge T
\end{equation}
is optimal for the associated OSP. Indeed, thanks to the regularity of the mapping $s\mapsto\phi(s) = \E[X_s]$, the obstacle given by $h(X_t, \phi(t))$ suits the classical theory, see for instance Theorem I.3 in \cite{bismut1977} or \cite{elkaroui79} for a detailed discussion on the topic.

In the recent papers \cites{djehicheeliehamadene, djehichedumitrescuzeng}, this result could be successfully applied to a large class of mean-field OSPs  whose value-process solves a mean-field reflected BSDEs i.e., satisfies the Bellman equation. That class of OSPs includes the following recursive OSP (we ignore the integral term)
\begin{equation*}
	Y_0=\underset{\tau\in \mathcal{T}_0}{\sup}\, \E\left[h(Y_{\tau},\P_{Y_{s}}\lvert_{s=\tau})\ind_{\{\tau<T\}}+\xi\ind_{\{\tau=T\}}\right].
\end{equation*}

Another example of a mean-field OSP for which the Bellman equation is preserved is considered in the recent work by Talbi, Touzi and Zhang \cite{ talbitouzizhang21} where the mean-field OSP for a mean-field diffusion is studied in a weak (or relaxed) formulation i.e. in terms of the joint marginal law of the stopped underlying process $X$ and the survival process $I_t:=\ind_{\{\tau>t\}}$ associated with the stopping time. Moreover, the performance function is a deterministic function of the marginal laws of $(X,I)$. Namely, given a probability measure$\mu$ with finite second moment,
$$
Y_0=\underset{\P}{\sup} \int_0^T F(s,\P_{(X_s,I_s)})ds+g(\P_{(X_T,I_T)})
$$
where under $\P$, the 'coordinate process' $(X_s,I_s)$ satisfies
$$
X_t=X_0+\int_0^t b(s,X_s,\P_{(X_s,I_s)})I_sds+\int_0^t \sigma(s,X_s,\P_{(X_s,I_s)})I_sdW^{\P}_s,\,\, I_t=I_{0^-}\ind_{\{\tau>t\}},
$$
under the constraint $\P_{X_0}=\mu,\quad \P(I_{0^-}=1)=1$, where $W^{\P}$ is a Brownian motion under $\P$. They characterized the value function by a dynamic programming equation on the Wasserstein space.

Going back to the OSP of the variance of a process $X$, it can be seen as an OSP where the stopped obstacle is of the form $L_{\tau}=(X_{\tau}-\E[X_{\tau}])^2$. In this case, the results in \cite{ talbitouzizhang21} allow to solve only the associated relaxed problem. Nevertheless, Pedersen and Peskir \cites{pedersen11, pedersenpeskir16}) could solve the infinite horizon OSP of the variance of an underlying Markov diffusion process $X$ starting at $x$ at time $t=0$, namely 
\begin{equation}\label{ped}
\tau^{*}(x)\in \underset{\tau}{\argmax}\, \E_x[(X_{\tau}-\E_x[X_{\tau}])^2],
\end{equation}
by embedding it into an auxiliary standard OSP whose value function solves a standard variational inequality. To do so, they exploit  the following simple but powerful variational characterization of the variance: for any stopping time $\tau$ such that $\E_x[X_{\tau}^2]$ is finite, 
$$
\E_x[X_{\tau}]=\arg\min_{a\in\R}\E_x[(X_{\tau}-a)^2].
$$
More specifically, Pedersen \cite{pedersen11} considers the auxiliary optimal stopping problem 
\begin{equation}\label{c-stop}
\underset{\tau}{\sup}\, \E_x[(X_{\tau}-c)^2]
\end{equation}
for a given constant $c$, whose value process is simply 
$$
Y^{(c)}_t=\underset{\tau\in\mathcal{T}_t}{\esssup}\, \E_x\left[(X_{\tau}-c)^2\,|\,\Fc_t\right].
$$
An optimal stopping time for that problem is 
$$
\tau^{(c)}:=\inf\{t>0, \,\, Y_t^{(c)}=(X_t-c)^2 \}.
$$
By the above characterization of the variance, if $c^*(x)$ is a constant such that the value function $Y^{(c^*(x))}_0$ of stopping problem \eqref{c-stop} is finite and the optimal stopping time $\tau^{(c^*(x)))}$ satisfies the matching condition
$$
c^*(x)=\E_x[X_{\tau^{(c^*(x))}}],
$$
then $\tau^*(x):=\tau^{(c^*(x)))}$ is optimal for the OSP \eqref{ped}.

Due to the presence of the term the expected value $\E_x[X_{\tau}]$ of the random variable $X_{\tau}$, the obtained optimal stopping times and the related stopping boundaries depend on the starting points $x$ of the process and therefore are 'pre-committed' in the terminology used in \cite{pedersenpeskir16}.

In the present paper, we consider the following class of finite horizon time-inconsistent mean-field OSPs beyond the mean-variance case.   
Let $T>0$ be a finite time horizon, $W$ a Brownian motion defined on a probability space $(\Omega, \mathcal{F},\P)$ and $\mathbb{F}=\{\mathcal{F}_t\}_{t\in[0,T]}$ the $\P$-completed Brownian filtration. 

For a certain $\mathcal{F}_T$-measurable final condition $\xi$ and a performance function $h$, we consider the following OSPs:
\begin{itemize}
    \item[(OSPa)] Optimal stopping of a  recursive utility function defined on $(\Omega, \mathcal{F},\mathbb{F},\P)$:
\begin{equation}\label{Y-0-intro}
Y_0=\underset{\tau\in \mathcal{T}_0}{\sup}\,\E\left[h(Y_{\tau},\E[Y_{\tau}])\ind_{\{\tau<T\}}+\xi\ind_{\{\tau=T\}}\right],
\end{equation} 
where
\begin{equation*}
    Y_t=\underset{\tau\in \mathcal{T}_t}{\esssup}\, \E\left[h(Y_{\tau},\E[Y_{\tau}])\ind_{\{\tau<T\}}+\xi\ind_{\{\tau=T\}}\,|\,\Fc_t\right],
\end{equation*}
which appears in the modeling of prospective reserves in life insurance, see \cites{djehicheeliehamadene}.

\item[(OSPb)] Optimal stopping of a mean-field diffusion:
    \begin{equation}\label{Y-0-d-intro}
	Y_0=\underset{\tau\in \mathcal{T}_0}{\sup}\,\E\left[h(X_{\tau},\E[X_{\tau}])\ind_{\{\tau<T\}}+\xi\ind_{\{\tau=T\}}\right],
\end{equation}
where $X$ is diffusion process of mean-field type:
\begin{equation*}
X_t=X_0+\int_0^tb(s,X_s, \E[X_s])ds+\int_0^t \sigma(s,X_s,\E[X_s])dW_s.
\end{equation*}
\end{itemize}

\medskip
The main purpose of the present work is to show under certain conditions that an optimal stopping time for each of the OSPs \eqref{Y-0-intro} and \eqref{Y-0-d-intro}
can be characterized as the first time the value process hits the obstacles  $h(Y_t,\E[Y_t])$ for the OSP \eqref{Y-0-intro} and $h(X_t,\E[X_t])$ for the OSP \eqref{Y-0-d-intro}.

A straightforward extension is to consider the combination of \eqref{Y-0-d-intro} and \eqref{Y-0-intro} given by
\begin{equation}\label{Y-0-comb-intro}
Y_0=\underset{\tau\in \mathcal{T}_0}{\sup}\,\E\left[h(X_{\tau},\E[X_{\tau}],Y_{\tau},\E[Y_{\tau}])\ind_{\{\tau<T\}}+\xi\ind_{\{\tau=T\}}\right].
\end{equation}

To solve the above problems we use a limit approach which consists of introducing a family of interacting Snell envelopes $\{Y^{i,n}\}_{i=1}^n$ (see Section \ref{sec-formulation} for a precise definition for (OSPa) and Section \ref{ssec-diffusion} for (OSPb)) as approximation of the value-process of the mean-field OSP. For instance, we approximate the OSP \eqref{Y-0-intro} with the following family of interacting OSPs:
\begin{equation*}
	Y^{i,n}_0=\underset{\tau\in \mathcal{T}^i_0}{\sup}\, \E\left[h(Y^{i,n}_{\tau},\frac{1}{n}\sum_{j=1}^n Y^{j,n}_{\tau})\ind_{\{\tau<T\}}+\xi^i\ind_{\{\tau=T\}}\right],\quad i=1,2,\ldots,n.
\end{equation*}
These problems are time-consistent and it can be shown (see Corollary \ref{opt-i-n}, below) that  it is optimal to stop at the hitting time $\hat{\tau}^{i,n}$ at which the value-process (which is now a Snell envelope) $(Y^{i,n}_t)_{t\geq 0}$ hits the barriere  $(\E[h(Y^{i,n}_{t},\frac{1}{n}\sum_{j=1}^n Y^{j,n}_{t})\, \lvert\,\mathcal{F}^i_t])_{t\in[0,T]}$. In Theorem  \ref{optimality-tau-i} below we prove that the stopping time
\begin{equation}\label{tau-opt}
	\tau^{*}=\inf\{t\ge 0, \,\, Y_t=h(Y_{t},\E[Y_{t}])\}\wedge T
\end{equation}
is optimal for $Y_0$ given by \eqref{Y-0-intro}, by showing that it is  the limit in probability of $\hat{\tau}^{1,n}$ as $n\to\infty$. Thus,  in this time-inconsistent framework an optimal stopping is also given by the usual hitting time. To derive this result, we need to investigate the convergence of $\{Y^{i,n}_0\}_{n\geq1}$ to $Y_0$ (Theorem  \ref{conv-1}) and the convergence of the associated optimal stopping times (Proposition \ref{conv-os-prob}).

\ms
As a final remark, we point out that by embedding this class of OSPs into the ones w.r.t. the set of randomized stopping times which is compact in the Baxter-Chacon topology (cf. \cite{baxter77}), following many papers including Edgar, Millet and Sucheston \cite{edgar82}, Arenas \cite{arenas90}, El Karoui, Lepeltier and Millet \cite{millet92}, and Pennanen and Perkki{\"o} \cite{pennanenperkkio}, it should be possible to show that there exists an optimal randomized stopping time for $Y_0$ without further characterization compared to the explicit optimal stopping time \eqref{tau-opt}. 

\medskip
The paper is organized as follows. In Section \ref{sec-formulation} we state precisely the problem (OSPa) and the assumptions we need for the remaining part of the section. Then we discuss the well-posedness of the studied system of interacting optimal stopping problems, which is not obvious due to the recursive form of the utility function. In Section \ref{sec-convergence} we present the main results of the paper, Theorem \ref{conv-1} about the convergence of the family of value-processes of time-consistent OSPs to the value-process of the time-inconsistent OSP and Theorem \ref{optimality-tau-i} about the convergence of related optimal stopping times. 
In Section \ref{ssec-diffusion} we discuss how to apply the suggested techniques to the problem (OSPb) associated to a mean-field  diffusion process. Finally, in Section \ref{osp-markov}, we  discuss the OSP of the variance of a Markov diffusion processes $X$ starting at $x$ at $t=0$.  We provide the main ingredients of the limit approach of Section \ref{ssec-diffusion}, which lead to the proof of optimality of the 'pre-committed' hitting time $\tau^*(x)$ give by  
\begin{equation}\label{tau-var-int}
	\tau^{*}(x)=\inf\{t\ge 0, \,\, Y_t(x)=(X_{t}-\E_x[X_{t}])^2\}\wedge T
\end{equation}
for
\begin{equation*}
    Y_0(x)=\underset{\tau\in \mathcal{T}_0}{\sup}\, \E_x\left[(X_{\tau}-\E_x[X_{\tau}])^2)\ind_{\{\tau<T\}}+\xi\ind_{\{\tau=T\}}\right],
\end{equation*}
where
\begin{equation*}
    Y_t(x)=\underset{\tau\in \mathcal{T}_t}{\esssup}\, \E_x\left[(X_{\tau}-\E_x[X_{\tau}])^2)\ind_{\{\tau<T\}}+\xi\ind_{\{\tau=T\}}\,|\,\Fc_t\right].
\end{equation*}
Although we cannot literally compare our finite-time horizon time-inconsistent OSP to the one studied in \cites{pedersen11,pedersenpeskir16}, they share the same feature of being first hitting times of the obstacle by the value process and for being 'pre-committed' optimal stopping times. A further characterization of the associate value function similar to the one provided in \cite{pedersen11} is not discussed in the present paper but deserves to be done in stand alone paper.

\medskip

Throughout this paper we only consider the one-dimensional Brownian motion and diffusion processes. The generalization to the multidimensional case is straightforward.

\medskip
Extension of the obtained results to general OSPs of a  recursive utility function associated with a mean-field diffusion of the form
\begin{equation*}\label{Y-0-d-intro-c}
	Y_0=\underset{\tau\in \mathcal{T}_0}{\sup}\,\E\left[h(X_{\tau},\P_{X_{\tau}}, Y_{\tau},\P_{Y_{\tau}})\ind_{\{\tau<T\}}+\xi\ind_{\{\tau=T\}}\right],
\end{equation*}
where
\begin{equation*}
X_t=X_0+\int_0^tb(s,X_s, \P_{X_{s}})ds+\int_0^t \sigma(s,X_s,\P_{X_{s}})dW_s,
\end{equation*}
can be done without difficulty at the cost of using heavier technical machinery. 

\subsection*{Notation.}
Let $(\Omega, \mathcal{F},\mathbb{P})$ be a complete probability space and $T>0$ a finite time horizon.  $W=(W_t)_{t\in[0,T]}$  is a standard one-dimensional Brownian motion. We denote by $\mathbb{F} = \{\mathcal{F}_t\}_{t\in[0,T]}$ the $\P$-completed natural filtration of the Brownian motion $W$, with $\mathcal{F}_0=\{\emptyset, \Omega\}$. In particular, $\mathbb{F}$ is continuous i.e. for each $t\ge 0$  $\mathcal{F}_{t^-}=\mathcal{F}_t$. Let $\mathcal{P}$ be the $\sigma$-algebra on $\Omega \times [0,T]$ of $\mathcal{F}_t$-progressively measurable sets. 
Next, we introduce the following spaces. \smallskip
\begin{itemize} 
    \item $\mathcal{T}_t$ is the set of $\mathbb{F}$-stopping times $\tau$ such that
    $\tau \in [t,T]$ a.s. \ms
    \item \textcolor{black}{$L^2(\mathcal{F}_T)$ is the set of random variables $\xi$ which are $\mathcal{F}_T$-measurable and $\mathbb{E}[|\xi|^2]<\infty$.} \ms
    \item $\mathcal{S}^2$ is the set of real-valued $\mathcal{P}$-measurable processes $y$ for which \newline $\|y\|^2_{\mathcal{S}^2} :=\mathbb{E}[\underset{ u\in[0,T]}{\sup} |y_u|^2]<\infty$.  \ms
    \item $\mathcal{S}_{c}^{2}$ is the space of $\mathcal{S}^{2}$-valued continuous  processes. This space is complete and separable.  \ms
    \item $C([0,T];\R)$ is the  space of continuous functions over $[0,T]$ endowed with the supremum norm. It is a separable Banach space.  
\end{itemize}
\medskip
\section{Optimal stopping of a  recursive utility function}\label{sec-formulation}
Let us introduce the recursive value-process
\begin{equation}\label{Y}
Y_t=\underset{\tau\in \mathcal{T}_t}{\esssup}\, \E\left[h(Y_{\tau},\E[Y_{\tau}])\ind_{\{\tau<T\}}+\xi\ind_{\{\tau=T\}}\,|\,\Fc_t\right].
\end{equation}
where $h$ is a sufficiently smooth cost (see Assumption \ref{A1} below) and  $\xi\in L^2(\Fc_T)$. The (simplified) finite horizon optimal stopping problem (OSP) of mean field type associated to \eqref{Y} reads as:
\begin{equation}\label{Y-0}
Y_0=\underset{\tau\in \mathcal{T}_0}{\sup}\, \E\left[h(Y_{\tau},\E[Y_{\tau}])\ind_{\{\tau<T\}}+\xi\ind_{\{\tau=T\}}\right].
\end{equation}

This class of OSP is motivated by  nonlinear prospective reserving models in life insurance. See the explicit example of Guaranteed life endowment with a surrender/withdrawal option described in \cite{djehicheeliehamadene}.

As mentioned in the introduction, the OSP \eqref{Y-0} is time-inconsistent i.e. the associated value-process does not satisfy the Bellman equation \eqref{Bellman}, due to the presence of expected value (law) of the random variable $Y_{\tau}$. We would like to investigate whether the value-process $Y$ is well-defined i.e. whether there exists a unique solution to \eqref{Y} and whether there exists a optimal stopping time $\tau^*$ to the OSP \eqref{Y-0}:
\begin{equation}\label{stop}
    \tau^*=\underset{\tau\in \mathcal{T}_0}{\arg\max}\, \E\left[h(Y_{\tau},\E[Y_{\tau}])\ind_{\{\tau<T\}}+\xi\ind_{\{\tau=T\}}\right].
\end{equation}

\medskip
We suggest to solve this problem by using a limit approach based on approximating $\E[Y_{\cdot}]$ by its empirical mean $\frac{1}{n}\sum_{j=1}^n Y^{j,n}_{\cdot}$ for some suitable sample $Y^{i,n},\,i=1,2,\ldots,n$ of 'interacting' value-processes which solve a system of standard OSPs.

\medskip
To this end we set $W^1=W$ and let $\{W^i\}_{i\ge 1}$ be independent Brownian motions and for each $i\ge 1$, denote by ${\mathbb{F}}^i:=\{{\mathcal{F}}^i_t\}_{t\in[0,T]}$ the $\P$-completion of the filtration generated by $W^i$. Let  $\mathcal{T}^i_t$ be the set of $\mathbb{F}^i$ stopping times with values in $[t,T]$. 

\medskip
Consider the following family of finite horizon stopping problems.
\begin{equation}\label{Y-i-n}
Y^{i,n}_t=\underset{\tau\in \mathcal{T}^i_t}{\esssup}\, \E\left[h(Y^{i,n}_{\tau},\frac{1}{n}\sum_{j=1}^n Y^{j,n}_{\tau})\ind_{\{\tau<T\}}+\xi^i\ind_{\{\tau=T\}}\,|\,\Fc^i_t\right],\quad i=1,2,\ldots,n,
\end{equation}
and
\begin{equation}\label{Y-i}
Y^{i}_t=\underset{\tau\in \mathcal{T}^i_t}{\esssup}\, \E\left[h(Y^{i}_{\tau},\E[Y^{i}_{\tau}])\ind_{\{\tau<T\}}+\xi^i\ind_{\{\tau=T\}}\,|\,\Fc^i_t\right],\quad i\ge 1,
\end{equation}
where $h$ and $\{\xi^i\}_{i\geq1}$ satisfies the following conditions.

\begin{Assumption}\label{A1} the sequence $\{\xi^i\}_{i\geq 1}$ and the function $h$  satisfy the following conditions:
\begin{enumerate}[(i)]
		\item For each $i\ge 1$, $\xi^i\in L^2(\Fc^i_T)$. Moreover, the $\xi^i$'s are independent copies of $\xi$, with $\xi^1=\xi$; \\
		\item the function  $h\colon \R\times\R\to\R$ \\
  is Lipschitz continuous w.r.t. $(y,z)$: there exist two positive constants $\ga$ and $\gb$ such that 
		\begin{equation*}
			\lvert h(y_1,z_1) - h(y_2,z_2)\rvert\leq \ga\lvert y_1 - y_2\rvert + \gb\lvert z_1 - z_2\rvert, 
		\end{equation*}
		for any $y_1,y_2,z_1,z_2\in\R$.
	\end{enumerate}
\end{Assumption}

\subsection{Existence and uniqueness of the value-processes}\label{sec-existence-uniq}
In this section we show the well-posedness of the systems \eqref{Y-i} and \eqref{Y-i-n}, since in both cases the performance function depends also on the value-process. 
\begin{theorem}\label{Y-i-well-def} Suppose that Assumption \ref{A1} is in force. Assume further that  $\ga$ and $\gb$ satisfy
	\begin{equation}\label{cond_gagb}
		\ga^2 + \gb^2<\frac{1}{2}.
	\end{equation}
Then there exists a unique solution in $\mathcal{S}_c^2$ to each of the systems \eqref{Y-i-n} and \eqref{Y-i}.
\end{theorem}

The proof of Theorem \eqref{Y-i-well-def} is based on a fixed point argument similar to the one used in the proof of Theorem 2.1 in \cite{djehichedumitrescuzeng}. We omit to reproduce it here. 

\begin{corollary}\label{opt-i-n} For each $i=1,\ldots n$, the $\Fc^i$-stopping time
\begin{equation}\label{tau-i-n}
    \hat{\tau}^{i,n}=\inf\left\{t\ge 0, \,\, Y^{i,n}_t=\E[h(Y^{i,n}_{t},\frac{1}{n}\sum_{j=1}^n Y^{j,n}_{t})\, \lvert\,\mathcal{F}^i_t]\right\}\wedge T
    \end{equation}
is optimal for $Y_0^{i,n}$.
\end{corollary}
\begin{proof} 
Since each of the processes $Y^{i,n}$ is in $\mathcal{S}_c^2$, by Assumption \ref{A1} (ii) and Doob's inequality, it follows that the obstacle process $\mathcal{X}^{i,n}_t:=\E[h(Y^{i,n}_{t},\frac{1}{n}\sum_{j=1}^n Y^{j,n}_{t})\, \lvert\,\mathcal{F}^i_t]$ is also in $\mathcal{S}^2$ and thus it is in the class $D$ of c{\`a}dl{\`a}g processes. Moreover, since for each $t\in [0,T]$, both $h(Y^{i,n}_{t},\frac{1}{n}\sum_{j=1}^n Y^{j,n}_{t})$ and $\mathcal{F}^i_t$ are continuous, the optional and predictable projections of $h(Y^{i,n}_{\cdot},\frac{1}{n}\sum_{j=1}^n Y^{j,n}_{\cdot})$ w.r.t. $\mathbb{F}^i$ coincide. This implies that the obstacle process $\mathcal{X}^{i,n}$ is a.s. continuous.  Therefore, by  e.g. Theorem D.12 in Appendix D in \cite{karatzasshreve98} or Proposition 1.1.8 in \cite{pham} (see also \cite{bismut1977} and \cite{elkaroui79} for a more general set up), for each $i=1,2,\ldots,n$, the stopping time $\tau^{i,n}$ given by \eqref{tau-i-n} is optimal for $Y^{i,n}_0$.
\end{proof}

\section{Convergence results}\label{sec-convergence}
The main aim of this section is to characterize an optimal stopping time for the time-inconsistent problem \eqref{Y-0}. The idea is to exploit the time consistency of the system of interacting optimal stopping problems
\begin{equation*}
	Y^{i,n}_0=\underset{\tau\in \mathcal{T}^i_0}{\sup}\, \E\left[h(Y^{i,n}_{\tau},\frac{1}{n}\sum_{j=1}^n Y^{j,n}_{\tau})\ind_{\{\tau<T\}}+\xi^i\ind_{\{\tau=T\}}\right],\quad i=1,2,\ldots,n,
\end{equation*} to obtain an explicit sequence of optimal stopping times, and then to show that this sequence converges to an optimal stopping time for the OSP \eqref{Y-0}.
\subsection{Convergence of the particle system}\label{ssec_convergence}
This subsection is concerned with the convergence of the process $Y^{i,n}$, solution of the particle system \eqref{Y-i-n}, to the solution $Y^{i}$ of the mean-field system \eqref{Y-i}. We first recall the notion of exchangeable random variables.

\begin{definition}[Exchangeable r.v.] The random variables $X^1,X^2,\dots,X^n$ are said to be exchangeable if the law of the random vector $(X^1,X^2,\dots,X^n)$ is the same as that of the random vector $(X^{\sigma(1)},X^{\sigma(2)},\dots,X^{\sigma(n)})$ for every permutation $\sigma$ of the set $\{1,2,\dots,n\}$. We write
\begin{equation*}
	{law}(X^1,X^2,\dots,X^n) = law(X^{\sigma(1)},X^{\sigma(2)},\dots,X^{\sigma(n)}).
\end{equation*}
\end{definition}
In the following proposition, we show that the exchangeability property of the final conditions $\{\xi^i\}_{i\geq1}$, entailed by Assumption \ref{A1} (i), transfers to the solutions of the systems \eqref{Y-i-n} and \eqref{Y-i}.
\begin{proposition}[Exchangeability property]\label{exchangeability} Let Assumption \ref{A1} hold and consider the sequence of processes $\{Y^{i,n}\}_{i=1}^n$ solution of the system \eqref{Y-i-n}. Then the processes $Y^{1,n},Y^{2,n},\ldots,Y^{i,n}$ are exchangeable.  Moreover, for every $n\ge 1$, the processes $Y^1,Y^2,\ldots,Y^n$, where each $Y^i$ is the solution of the system \eqref{Y-i}, are independent and equally distributed and hence exchangeable.
\end{proposition}
\begin{proof}
	First let us focus on  $\{Y^{i,n}\}_{i=1}^n$. For any permutation $\sigma$ of the set $\{1,2,\dots,n\}$, we have for any $t\in[0,T]$ 
	\begin{equation*}
		\frac{1}{n}\sum_{j=1}^n Y^{j,n}_t = \frac{1}{n}\sum_{j=1}^n Y^{\sigma(j),n}_t,
	\end{equation*}
	and thanks to the uniqueness result in Theorem \ref{Y-i-well-def} we have
	\begin{equation*}
		\text{law}(Y^{1,n},Y^{2,n},\dots,Y^{n,n}) = \text{law}(Y^{\sigma(1),n},Y^{\sigma(2),n},\dots,Y^{\sigma(n),n}), 
	\end{equation*}
 i.e., the processes $\{Y^{i,n}\}_{i=1}^n$ are exchangeable.
	Regarding the processes $\{Y^{i}\}_{i\geq1}$, from \eqref{Y-i} we may write $Y_{t}^i=\varphi((W^i_{s})_{0\le s\le t})$ for some Borel measurable function $\varphi$. But, the $(W^i)$'s  are independent and  equally distributed. Therefore, the $Y^i$'s are independent and equally distributed and thus exchangeable.
\end{proof}

Theorem \ref{conv-1} below is the first main result of the paper. It shows convergence of the system of interacting Snell envelops $\{Y^{j,n}\}_{n\ge 1}$ to the time-inconsistent value processes $Y^j$ in $\Sc^2$. Due to time-inconsistency caused by the terms $\E[Y^j_{\tau}]$, the proof does not trivially follow from standard $L^2$-estimates and the Lipschitz continuity of $h$. As we will see it below, the proof is completed thanks to the estimate \eqref{2-app}, which we could not find in the literature. We note that the smallness condition \eqref{cond_gagb_3} in the statement of Theorem \ref{conv-1} appears natural from the calculations, but is not the optimal one. It can definitely be improved.
\begin{theorem}\label{conv-1}
	Assume that $\ga$ and $\gb$ satisfy 
	\begin{equation}\label{cond_gagb_3}
		\ga^2 + \gb^2<\frac{1}{16}. 
	\end{equation}
	Then, under Assumption \ref{A1} we have
	\begin{equation*}
		\Lim_{n\to\infty}\underset{1\le i \le n}{\sup}\, \E\left[\sup_{t\in[0,T]}\lvert Y^{i,n}_t - Y^{i}_t\rvert^2\right] = 0.
	\end{equation*}
\end{theorem}
\begin{proof}
For any $t\leq T$, we have
\begin{equation}\label{Y-i-n-est}
\begin{aligned}
	\lvert Y^{i,n}_t - Y^i_t \rvert& = \Bigg\lvert \underset{\tau\in \mathcal{T}^i_t}{\esssup}\,\E\left[h(Y^{i,n}_{\tau},\frac{1}{n}\sum_{j=1}^n Y^{j,n}_{\tau})\ind_{\{\tau<T\}}+\xi^i\ind_{\{\tau=T\}}\,|\,\Fc^i_t\right] \\
	&\quad- \underset{\tau\in \mathcal{T}^i_t}{\esssup}\, \E\left[h(Y^i_{\tau},\E[Y^i_{\tau}])\ind_{\{\tau<T\}}+\xi^i\ind_{\{\tau=T\}}\,|\,\Fc^i_t\right] \Bigg\rvert\\
	&\leq  \underset{\tau\in \mathcal{T}^i_t}{\esssup}\, \Bigg \lvert \E\left[h(Y^{i,n}_{\tau},\frac{1}{n}\sum_{j=1}^n Y^{j,n}_{\tau})\ind_{\{\tau<T\}}+\xi^i\ind_{\{\tau=T\}}\,|\,\Fc^i_t\right] \\
	&\quad-\E\left[h(Y^i_{\tau}, \E[Y^i_{\tau}])\ind_{\{\tau<T\}}+\xi^i\ind_{\{\tau=T\}}\,|\,\Fc^i_t\right]\Bigg\rvert\\
	&\leq \underset{\tau\in \mathcal{T}^i_t}{\esssup}\, \E\left[\left\lvert h(Y^{i,n}_{\tau},\frac{1}{n}\sum_{j=1}^n Y^{j,n}_{\tau}) - h(Y^i_{\tau}, \E[Y^i_{\tau}])\right\rvert\,|\,\Fc^i_t\right]\\
	&\leq \underset{\tau\in \mathcal{T}^i_t}{\esssup}\, \E\left[\left(\ga\lvert Y^{i,n}_{\tau} - Y^i_{\tau} \rvert + \gb\lvert \frac{1}{n}\sum_{j=1}^n Y^{j,n}_{\tau} - \E[Y^i_{\tau}] \rvert\right)\,|\,\Fc^i_t\right]
 \end{aligned}
\end{equation}
 But, by Lemma \ref{stop-app} from the appendix below, we have
 \begin{equation}\label{1-app}
 \underset{\tau\in \mathcal{T}^i_0}{\esssup}\,\lvert Y^{i,n}_{\tau} - Y^i_{\tau} \rvert =\underset{s\in [0,T]}{\sup}\, \lvert Y^{i,n}_{s} - Y^i_{s} \rvert. 
 \end{equation} 
Furthermore, it is tempting to claim that 
 \begin{equation*}
\underset{\tau\in \mathcal{T}^i_0}{\esssup}\,\lvert \frac{1}{n}\sum_{j=1}^n Y^{j,n}_{\tau} - \E[Y^i_{\tau}] \rvert\le \underset{s\in [0,T]}{\sup}\,\lvert \frac{1}{n}\sum_{j=1}^n Y^{j,n}_{s} - \E[Y^i_{s}] \rvert \quad \as.
 \end{equation*}
 This inequality is not always true since it implicitly claims that the norm $\underset{\tau\in \mathcal{T}^i_0}{\sup}\,\E[\lvert Y^i_{\tau}\rvert]$ is equivalent to the norm $\underset{t\in [0,T]}{\sup}\,\E[\lvert Y^i_t\rvert]$. But, a counter-example in \cite{dellacheriemeyer82}, pp. 82, shows that the norm   $\underset{\tau\in \mathcal{T}^i_0}{\sup}\,\E[\lvert Y^i_{\tau}\rvert]$  is much stronger. An equivalence between these norms holds if $Y^i$ is a martingale or $|Y^i|$ is supermartingale in which case  Doob's maximal inequality yields the equivalence. In our case $Y^i$ is typically not a martingale-like process. But, since it is adapted to the Brownian filtration, by using the martingale representation theorem we will show that 
\begin{equation}\label{2-app}
    \underset{\tau\in \mathcal{T}^i_0}{\esssup}\,\lvert \frac{1}{n}\sum_{j=1}^n Y^{j,n}_{\tau} - \E[Y^i_{\tau}] \rvert\le \underset{\alpha \in \mathcal{M}^i}{\esssup}\,\underset{s\in [0,T]}{\sup}\,\lvert \frac{1}{n}\sum_{j=1}^n Y^{j,n}_{s} - \E[Y^i_{s}] +M^i_s(\alpha) \rvert,
 \end{equation}
 where $\mathcal{M}^i$ will be determined below as a subset the set $\mathcal{P}^i$ of $\mathbb{F}^i$-progressively measurable process $(\alpha_s)_{0\le s\le T}$ such that $$
 \E[\int_0^T|\a_s|^2ds] <\infty,
 $$
and
 $$
 M^i_t(\alpha):=\int_0^t\alpha_s dB^i_s
 $$ is a uniformly integrable Brownian $\F^i$-martingale with mean zero. Indeed, for every $\tau\in\Tc^i_0$, the process $(m^i_t)_{0\le t\le T}$ defined by 
 \begin{equation}\label{m-u-0}
 m^i_t:=\E[Y^i_{\tau}\,\lvert \Fc^i_t]=\ind_{\{\tau\le t\}}Y^i_{\tau}+\ind_{\{\tau > t\}}\E[Y^i_{\tau}\,\lvert \Fc^i_t]
 \end{equation}
 is a continuous and uniformly integrable $\mathbb{F}^i$-martingale with $m^i_T=Y^i_{\tau}$ and $m^i_0=\E[Y^i_{\tau}]$. Therefore,  by the martingale representation theorem, there exists a unique process $u^i\in\mathcal{P}^i$ such that
 \begin{equation}\label{m-u-1}
 m^i_t=m^i_0+\int_0^t u^i_s dB^i_s.
 \end{equation}
So,
 \begin{equation}\label{m-u-2}
 \E[Y^i_{\tau}]=Y^i_{\tau}-\int_0^{T} u^i_s dB^i_s.
 \end{equation}
 From \eqref{m-u-0} and \eqref{m-u-1}, we obtain
 \begin{equation}\label{m-u-est}
\E[\underset{t\in [0,T]}{\sup}\,\lvert \int_0^t u^i_s dB^i_s\rvert^2]\le 4\E[\underset{t\in [0,T]}{\sup}\,\lvert Y^i_t\rvert^2].
 \end{equation}
 Since the $Y^i$'s are i.i.d., we have the uniform bound 
 \begin{equation}\label{m-u}
\underset{i\ge 1}{\sup}\,\E[\underset{t\in [0,T]}{\sup}\,\lvert \int_0^t u^i_s dB^i_s\rvert^2]\le 4\E[\underset{t\in [0,T]}{\sup}\,\lvert Y^1_t\rvert^2].
 \end{equation}
Upon conditioning on $\Fc_{\tau}$, we have
 \begin{equation}\label{tau-mg}
\E[Y^i_{\tau}]=Y^i_{\tau}-\int_0^{\tau} u^i_s dB^i_s.
 \end{equation}

\indent Similarly, by considering the continuous and uniformly integrable $\mathbb{F}^i$-martingale defined, for each fixed $t\in [0,T]$, by 
 $$
 \hat{m}^i_s:=\E[Y^i_t\,\lvert \Fc^i_s]=\ind_{\{t\le s\}}Y^i_t+\ind_{\{t > s\}}\E[Y^i_t\,\lvert \Fc^i_s]
 $$
 with $\hat{m}^i_T=Y^i_t$ and $\hat{m}^i_0=\E[Y^i_t]$,  there exists a unique process $v^i\in\mathcal{P}^i$ such that 
 \begin{equation*}
 Y^i_t=\E[Y^i_t]+\int_0^T v^i_sdB^i_s.
 \end{equation*}
Again, since $(\hat{m}_t)_t$ is a uniformly integrable martingale, by conditioning on $\Fc_{t}$, we obtain
 \begin{equation}\label{t-mg}
\E[Y^i_{t}]=Y^i_{t}-\int_0^{t} v^i_s dB^i_s.
\end{equation}
Note that this does not mean that $Y^i$ is a martingale since $\E[Y^i_{s}]\neq \E[Y^i_{t}]$ if $s\neq t$. \\
Moreover,
\begin{equation}\label{m-hat-v}
\underset{i\ge 1}{\sup}\,\E[\underset{t\in [0,T]}{\sup}\,\lvert \int_0^t v^i_s dB^i_s\rvert^2]\le 4\E[\underset{t\in [0,T]}{\sup}\,\lvert Y^1_t\rvert^2].
 \end{equation}
In view \eqref{tau-mg}, we have 
\begin{equation*}\label{tau-t-sum-1}
    \lvert \frac{1}{n}\sum_{j=1}^n Y^{j,n}_{\tau} - \E[Y^i_{\tau}] \rvert =\lvert \frac{1}{n}\sum_{j=1}^n Y^{j,n}_{\tau} - Y^i_{\tau}+\int_0^{\tau} u^i_s dB^i_s\rvert 
    \end{equation*}
    and by Lemma \ref{stop-app} it holds that
    \begin{equation*}
\lvert\frac{1}{n}\sum_{j=1}^n Y^{j,n}_{\tau} - Y^i_{\tau}+\int_0^{\tau} u^i_s dB^i_s\rvert  \le  \underset{t\in [0,T]}{\sup}\,\lvert \frac{1}{n}\sum_{j=1}^n Y^{j,n}_{t} - Y^i_{t}+\int_0^{t} u^i_s dB^i_s\rvert.
 \end{equation*}
 Finally, by \eqref{t-mg}, we arrive at 
 \begin{equation*}\label{tau-u}
    \lvert \frac{1}{n}\sum_{j=1}^n Y^{j,n}_{\tau} - \E[Y^i_{\tau}] \rvert \le  \underset{t\in [0,T]}{\sup}\,\lvert \frac{1}{n}\sum_{j=1}^n Y^{j,n}_{t} - \E[Y^i_{t}]+ M^i_t(a^i)\rvert,
 \end{equation*}
 where with $a^i:=u^i-v^i \in \mathcal{P}^i$
 \begin{equation}\label{M}
 M^i_t(a^i):=\int_0^{t} a^i_s dB^i_s
 \end{equation}
 which is a uniformly integrable Brownian $\mathbb{F}^i$-martingale. By taking $\mathcal{M}^i$ to be the subset of $\mathcal{P}^i$ of processes $\alpha^i:=u^i-v^i$ where $u^i$ are given by \eqref{tau-mg} and $v^i$ are given by \eqref{t-mg},
 we obtain \eqref{2-app}. \\ 
\indent We note that by \eqref{m-u} and \eqref{m-hat-v} we have the following uniform bound 
\begin{equation}\label{M-a-1}
\underset{ \a\in \mathcal{M}^i}{\esssup}\,\underset{t\in [0,T]}{\sup}\,\lvert M^i(\a)\rvert\le 2\underset{t\in [0,T]}{\sup}\,\lvert Y^i_t\rvert+2 \E[\underset{t\in [0,T]}{\sup}\,\lvert Y^i_t\rvert],\quad \as,\quad  i\ge 1.
 \end{equation}
 Thus,
\begin{equation}\label{M-a}
\underset{i\ge 1}{\sup}\,\E[\underset{ \a\in \mathcal{M}^i}{\esssup}\,\underset{t\in [0,T]}{\sup}\,\lvert M^i(\a)\rvert^2]\le 8\E[\underset{t\in [0,T]}{\sup}\,\lvert Y^1_t\rvert^2].
 \end{equation}

 \medskip
 By \eqref{1-app} and \eqref{2-app}, we have
 \begin{equation*}
\begin{aligned}
	\lvert Y^{i,n}_t - Y^i_t \rvert & \leq \E\left[\ga\sup_{s\in[0,T]}\lvert Y^{i,n}_{s} - Y^i_{s} \rvert + \gb\underset{\alpha \in \mathcal{M}^i}{\esssup}\,\underset{s\in[0,T]}{\sup}\,\lvert \frac{1}{n}\sum_{j=1}^n Y^{j,n}_{s} - \E[Y^i_{s}]+M^i_s(\alpha) \rvert\,|\,\Fc^i_t\right]
	\\
	&\leq \E\left[	G^{i,n}+\gb\Lambda^i_n\,\lvert\,\Fc^i_t\right],
 \end{aligned}
\end{equation*}
where
\begin{equation}\label{G-n}
	G^{i,n}:= \ga\underset{s\in [0,T]}{\sup}\,\lvert Y^{i,n}_{s} - Y^i_{s} \rvert +  \frac{\gb}{n}\sum_{j=1}^n \underset{s\in [0,T]}{\sup}\,\lvert Y^{j,n}_{s}-Y^j_{s}\rvert
 \end{equation}
 and
 \begin{equation}\label{Lambda-n}
 \Lambda^i_n:=\underset{s\in [0,T]}{\sup}\,\lvert \frac{1}{n}\sum_{j=1}^n Y^{j}_{s} - \E[Y^i_{s}] \rvert +\underset{\alpha \in \mathcal{M}^i}{\esssup}\,\underset{s\in [0,T]}{\sup}\,\lvert M^i_s(\alpha)\rvert.
\end{equation}
 By Doob's inequality, we have
\begin{equation*}
	\E[\sup_{t\in[0,T]} \lvert  Y^{i,n}_t - Y^i_t  \rvert^2]\leq 4\E\left[\left(G^{i,n}+\gb\Lambda^i_n\right)^2\right].
\end{equation*}
Therefore,
\begin{equation}\label{sup_ineq}
\begin{aligned}
	\E[\sup_{t\in[0,T]} \lvert  Y^{i,n}_t - Y^i_t  \rvert^2] &\leq 8\gb^2\E\left[(\Lambda^i_n)^2
\right] + 16\E\left[\ga^2\sup_{s\in[0,T]}\lvert Y^{i,n}_{s} - Y^i_{s} \rvert^2  \right.  \\   & \qquad\qquad\qquad\qquad\qquad+ \left. \frac{\gb^2}{n}\sum_{j=1}^n \sup_{s\in[0,T]}\lvert Y^{j,n}_{s} - Y^{j}_{s}\rvert^2\right].
\end{aligned}
\end{equation}
Since the processes $\{Y^{i,n}\}_{i=1}^n$ and $\{Y^i\}_{i\geq1}$ (see Proposition \ref{exchangeability}) are exchangeable,  we have 
\begin{equation}\label{exch-1}
\E\left[\frac{1}{n}\sum_{j=1}^n \sup_{s\in[0,T]}\lvert Y^{j,n}_{s} - Y^{j}_{s}\rvert^2\right]=\E\left[\sup_{t\in[0,T]} \lvert  Y^{i,n}_t - Y^i_t  \rvert^2\right].
\end{equation}
Thus, from \eqref{sup_ineq} we obtain
\begin{equation*}
	\E[\sup_{t\in[0,T]} \lvert  Y^{i,n}_t - Y^i_t  \rvert^2]\leq16(\ga^2+\gb^2)\E[\sup_{s\in[0,T]}\lvert Y^{i,n}_{s} - Y^i_{s} \rvert^2] +  8\gb^2\E[(\Lambda^i_n)^2],
\end{equation*}
where, by \eqref{cond_gagb_3}, $16(\ga^2+\gb^2)<1$. Furthermore,
since, by \eqref{Lambda-n},
\begin{equation}\label{Lambda-n-1}
    \E[(\Lambda^i_n)^2]\le 2 \E[\underset{s\in [0,T]}{\sup}\,\lvert \frac{1}{n}\sum_{j=1}^n Y^{j}_{s} - \E[Y^j_{s}] \rvert ^2] +2\E[\underset{\alpha \in \mathcal{M}^i}{\esssup}\,\underset{s\in [0,T]}{\sup}\,\lvert M^i_s(\alpha)\rvert^2],
\end{equation}
 with 
$C_{\g}:=16(1-16(\ga^2+\gb^2))^{-1}$, we have
\begin{equation*}
	\E[\sup_{t\in[0,T]} \lvert  Y^{i,n}_t - Y^i_t  \rvert^2]\leq C_{\g}\E[\underset{s\in [0,T]}{\sup}\,\lvert \frac{1}{n}\sum_{j=1}^n Y^{j}_{s} - \E[Y^j_{s}] \rvert ^2] +C_{\g}\E[\underset{\alpha \in \mathcal{M}^i}{\esssup}\,\underset{s\in [0,T]}{\sup}\,\lvert M^i_s(\alpha)\rvert^2].
\end{equation*} 
Thus, in view of \eqref{exch-1}, we obtain
\begin{equation}\label{est-fund}\begin{aligned}
	\underset{1\le i \le n}{\sup}\,\E[\sup_{t\in[0,T]} \lvert  Y^{i,n}_t - Y^i_t  \rvert^2] & \leq C_{\g}\E[\underset{s\in [0,T]}{\sup}\,\lvert \frac{1}{n}\sum_{j=1}^n Y^{j}_{s} - \E[Y^j_{s}] \rvert ^2] \\ & \qquad\qquad +C_{\g}\frac{1}{n}\sum_{j=1}^n \E[\underset{\alpha \in \mathcal{M}^j}{\esssup}\,\underset{s\in [0,T]}{\sup}\,\lvert M^j_s(\alpha)\rvert^2].
 \end{aligned}
\end{equation} 
By the strong law of large numbers for i.i.d. $C([0,T];\R)$-valued random variables with finite second moments (see Theorem 4.1.1 in \cite{padgett-taylor06}) and Dominated Convergence, we have
    \begin{equation}\label{SLLN-1}
    \Lim_{n\to\infty}\E[\underset{s\in [0,T]}{\sup}\,\lvert \frac{1}{n}\sum_{j=1}^n Y^{j}_{s} - \E[Y^j_{s}] \rvert ^2]=0.
    \end{equation}
    It remains to show that 
     \begin{equation}\label{SLLN-2}
    \Lim_{n\to\infty}\frac{1}{n}\sum_{j=1}^n \E[\underset{\alpha \in \mathcal{M}^j}{\esssup}\,\underset{s\in [0,T]}{\sup}\,\lvert M^j_s(\alpha)\rvert^2]=0.
\end{equation}
Now, since by \eqref{M-a} the sequence of independent Brownian martingales $\{M^j(\alpha^j)\}_{j\ge 1}$ is tight, we have
\begin{equation}\label{SLLN-3}
    \Lim_{n\to\infty}\frac{1}{n}\sum_{j=1}^n \E[\,\underset{s\in [0,T]}{\sup}\,\lvert M^j_s(\alpha^j)\rvert^2]=0.
    \end{equation}
Indeed, by Doob's inequality and the strong law of large numbers for the sequence of tight, independent and centered r.v. $\{ M_T^j(\a^j)\}_{j\ge 1}$ along with Dominated Convergence we have
$$
\Lim_{n\to\infty}\frac{1}{n}\sum_{j=1}^n \E[\,\underset{s\in [0,T]}{\sup}\,\lvert M^j_s(\alpha^j)\rvert^2]\le 4\Lim_{n\to\infty}\frac{1}{n}\sum_{j=1}^n \E[\lvert M^j_T(\alpha^j)\rvert^2]=0.
$$
We shall use \eqref{SLLN-3} to derive \eqref{SLLN-2}. Indeed, by the properties of the essential supremum, for each $j\ge 1$, there exists a sequence $\{\a^j_m\}_{m\ge 1}$ in $\mathcal{M}^j$ such that 
$$
\underset{\alpha \in \mathcal{M}^j}{\esssup}\,\underset{s\in [0,T]}{\sup}\,\lvert M^j_s(\alpha)\rvert^2=\Lim_{m\to\infty}\underset{s\in [0,T]}{\sup}\,\lvert M^j_s(\alpha^j_m)\rvert^2\quad\as.
$$
By Dominated Convergence, we have
$$
\E[\underset{\alpha \in \mathcal{M}^j}{\esssup}\,\underset{s\in [0,T]}{\sup}\,\lvert M^j_s(\alpha)\rvert^2]=\Lim_{m\to\infty}\E[\underset{s\in [0,T]}{\sup}\,\lvert M^j_s(\alpha^j_m)\rvert^2]\le \underset{m\ge 1}{\sup}\,\E[\underset{s\in [0,T]}{\sup}\,\lvert M^j_s(\alpha^j_m)\rvert^2].
$$
We claim that
\begin{equation}\label{sum-sup}
   \Lim_{n\to\infty}\frac{1}{n}\sum_{j=1}^n  \underset{m\ge 1}{\sup}\,\E[\underset{s\in [0,T]}{\sup}\,\lvert M^j_s(\alpha^j_m)\rvert^2]=0.
\end{equation}
If this would not be the case, then there would exist a $\delta>0$ such that for all $n_0\ge 1$, there would exist an $n\ge n_0$ such that 
$$
\frac{1}{n}\sum_{j=1}^n  \underset{m\ge 1}{\sup}\,\E[\underset{s\in [0,T]}{\sup}\,\lvert M^j_s(\alpha^j_m)\rvert^2]\ge \delta.
$$
But, for every $j\ge1$, there exists an $m_j\ge 1$ such that 
$$
\underset{m\ge 1}{\sup}\,\E[\underset{s\in [0,T]}{\sup}\,\lvert M^j_s(\alpha^j_m)\rvert^2]\ge \E[\underset{s\in [0,T]}{\sup}\,\lvert M^j_s(\alpha^j_{m_j})\rvert^2]\ge \frac{1}{2}\,\underset{m\ge 1}{\sup}\,\E[\underset{s\in [0,T]}{\sup}\,\lvert M^j_s(\alpha^j_m)\rvert^2],
$$
which entails 
$$
\frac{1}{n}\sum_{j=1}^n  \,\E[\underset{s\in [0,T]}{\sup}\,\lvert M^j_s(\alpha^j_{m_j})\rvert^2]\ge \frac{\delta}{2}.
$$
But, this  contradicts \eqref{SLLN-3}. This finishes the proof of the theorem.
\end{proof}

\subsection{Convergence of the optimal stopping times}\label{ssec-conv-ost}
In Section \ref{ssec_convergence} we proved that $Y^{i,n}$ converges to $Y^i$ as $n$ goes to  infinity in the $\mathcal{S}^2$ norm. This entails the convergence of the values of the OSPs $Y^{i,n}_0$ to $Y^i_0$, as $n\to\infty$, for every $i\geq1$. Furthermore, by Corollary \ref{opt-i-n}, for each $i=1,\ldots,n$, the stopping time $\hat{\tau}^{i,n}$ given by \eqref{tau-i-n} is optimal for the OSP $Y^{i,n}_0$, that is
\begin{equation}
    \hat{\tau}^{i,n} =\underset{\tau\in \mathcal{T}^i_0}{\arg\max}\, \E\left[h(Y^{i,n}_{\tau},\frac{1}{n}\sum_{j=1}^n Y^{j,n}_{\tau})\ind_{\{\tau<T\}}+\xi^i\ind_{\{\tau=T\}}\right].
\end{equation}

For every $i\geq 1$, let us introduce the stopping time
\begin{equation}\label{tau-i}
    \hat{\tau}^{i}=\inf\{t\ge 0, \,\, Y^{i}_t=h(Y^{i}_{t},\E[Y^i_{t}])\}\wedge T.
\end{equation}
Due to the time-inconsistency of the problem \eqref{Y-i}, it is not immediate to conclude that $\hat{\tau}^{i}$ is optimal, i.e. that it coincides with the optimal stopping $\tau^{i,*}$ defined by 
\begin{equation}\label{stop-i}
    \tau^{i,*}=\underset{\tau\in \mathcal{T}_0}{\arg\max}\, \E\left[h(Y^i_{\tau},\E[Y^i_{\tau}])\ind_{\{\tau<T\}}+\xi^i\ind_{\{\tau=T\}}\right].
\end{equation}
The main result of this section is to show that \eqref{stop-i} actually holds.
In particular, since $W^1=W$ and $\xi^1=\xi$ , we have $Y^1=Y$ i.e. $Y^1$ is the value-process $Y$ given by \eqref{Y} and $\tau^{1,*}$ is the associated optimal stopping time given by \eqref{stop} i.e. $\tau^*:=\tau^{1,*}$.
\begin{theorem}\label{optimality-tau-i}
	Let Assumption \ref{A1} hold and assume that $\ga$ and $\gb$ satisfy \eqref{cond_gagb_3}.
	Then for every $i\geq 1$ the stopping time $\hat{\tau}^i$ defined by \eqref{tau-i} is optimal for the OSP \eqref{stop-i}.
\end{theorem}

To prove this statement, we rely on the Proposition \ref{conv-os-prob} below about the convergence of optimal stopping times and its Corollary \ref{conv-os-prob-1}.  We also need the following 
\begin{lemma}\label{Z-n}
Set
    \begin{equation}\label{Z}
        Z^{i,n}_t:= Y^{i,n}_t - \E[h(Y^{i,n}_{t},\frac{1}{n}\sum_{j=1}^n Y^{j,n}_{t})\, \lvert\,\mathcal{F}^i_t],\quad Z^i_t = Y^{i}_t - h(Y^{i}_t,\E[Y^i_t]),\quad t\in[0,T].
    \end{equation}
Then, we have 
\begin{equation}\label{conv-mean-sup}
        \Lim_{n\to\infty}\underset{1\le i\le n} {\sup}\,\E\left[\sup_{t\in[0,T]}\lvert Z^{i,n}_t - Z^i_t\rvert^2\right]= 0.
    \end{equation}
\end{lemma}
\begin{proof}
By Assumption \ref{A1} (ii), for every $i=1,\ldots,n$, we have 
    \begin{equation*}
    \begin{aligned}
     \sup_{t\in[0,T]}\lvert Z^{i,n}_t - Z^i_t\rvert &
    \leq \sup_{t\in[0,T]}\lvert Y^{i,n}_t - Y^i_t\rvert + \sup_{t\in[0,T]}\lvert \E[h(Y^{i,n}_{t},\frac{1}{n}\sum_{j=1}^n Y^{j,n}_{t})\, \lvert\,\mathcal{F}^i_t] - h(Y^i_t,\E[Y^i_t])\rvert \\
    &\leq (1 + \ga)\sup_{t\in[0,T]}\lvert Y^{i,n}_t - Y^i_t\rvert + \gb \sup_{t\in[0,T]}\E[\sup_{s\in[0,T]}\lvert\frac{1}{n}\sum_{j=1}^n Y^{j,n}_{s} - \E[Y^i_s] \rvert\,\lvert\, \Fc_t^ i].
    \end{aligned}
    \end{equation*}
Using Doob's inequality, we obtain
    \begin{equation*}
    \begin{aligned}
\E\left[\left( \sup_{t\in[0,T]}\E[\sup_{s\in[0,T]}\lvert\frac{1}{n}\sum_{j=1}^n Y^{j,n}_{s} - \E[Y^i_s] \rvert\,\lvert\, \Fc_t^ i] \right)^2\right] \le 4 \E\left[\left(\sup_{t\in[0,T]}\lvert\frac{1}{n}\sum_{j=1}^n Y^{j,n}_{t} - \E[Y^i_t] \rvert \right)^2\right] \\ \le 8 \E\left[\left(\frac{1}{n}\sum_{j=1}^n \sup_{t\in[0,T]}\lvert Y^{j,n}_{t} -Y^j_t\rvert\right)^2\right]+8\E\left[\left(\sup_{t\in[0,T]}\lvert\frac{1}{n}\sum_{j=1}^n \lvert Y^{j}_{t} - \E[Y^j_t]\lvert\right)^2\right] \\ \le  8\left[  \sup_{t\in[0,T]}\lvert Y^{i,n}_t - Y^i_t\rvert^2\right] + 8\E[\widetilde{\Lambda}^2_n],
    \end{aligned}
    \end{equation*}
    where the first term of the last inequality  follows from the Cauchy-Schwarz inequality and the exchangeability of the processes $(Y^{1,n},Y^{2,n},\ldots,Y^{n,n})$ and $(Y^1,Y^2,\ldots,Y^n)$ (by Proposition \ref{exchangeability}) and $\widetilde{\Lambda}_n$ is given by 
    $$
    \widetilde{\Lambda}_n:=\sup_{t\in[0,T]}\lvert\frac{1}{n}\sum_{j=1}^n \lvert Y^{j}_{t} - \E[Y^j_t]\lvert.
    $$
    By the strong law of large numbers for i.i.d. $C([0,T];\R)$-valued random variables with finite second moments (see Theorem 4.1.1 in \cite{padgett-taylor06}) and Dominated Convergence, we have
    \begin{equation}\label{tilde-Lambda}
    \Lim_{n\to\infty}\E[\widetilde{\Lambda}_n^2]=0.
    \end{equation}
    Therefore, we have
   \begin{equation*}
   \begin{aligned}
   \underset{1\le i\le n} {\sup}\,	\E\left[ \sup_{t\in[0,T]}\lvert Z^{i,n}_t -Z^i_t\rvert^2\right]\leq 2\left((1 + \ga)^2 + 8\gb^2\right)\underset{1\le i\le n} {\sup}\,\E\left[ \sup_{t\in[0,T]}\lvert Y^{i,n}_t - Y^i_t\rvert^2\right] + 16\gb^2\E[\Lambda_n].
   \end{aligned}
   \end{equation*}
    Thus, thanks to Theorem  \ref{conv-1} and \eqref{tilde-Lambda} it holds that
    \begin{equation}\label{conv-mean-sup-Z}
        \Lim_{n\to\infty}\underset{1\le i\le n} {\sup}\,\E\left[\sup_{t\in[0,T]}\lvert Z^{i,n}_t - Z^i_t\rvert^2\right]= 0.
    \end{equation}
\end{proof}
\begin{proposition}\label{conv-os-prob}
    Let $\{\hat{\tau}^{i,n}\}_{n\geq 1}$ be the sequence of optimal stopping times defined by \eqref{tau-i-n} and let $\hat{\tau}^{i}$ be defined by \eqref{tau-i}. Let Assumption \ref{A1} and the small condition \eqref{cond_gagb_3} hold. Then for every $\eps>0$, 
    \begin{equation}\label{stop-i-n}
\Lim_{n\to \infty}\underset{1\le i\le n}{\sup}\,\P(|\hat{\tau}^{i,n}-\hat{\tau}^{i}|>\eps)=0.
\end{equation}
In particular, for every fixed $i\ge 1$, $\hat{\tau}^{i,n}$ converges to $\hat{\tau}^{i}$ in probability, as $n$ goes to infinity.
\end{proposition}
\begin{proof}
Recall the notation
    \begin{equation*}
        Z^{i,n}_t:= Y^{i,n}_t - \E[h(Y^{i,n}_{t},\frac{1}{n}\sum_{j=1}^n Y^{j,n}_{t})\, \lvert\,\mathcal{F}^i_t],\quad Z^i_t = Y^{i}_t - h(Y^{i}_t,\E[Y^i_t]),\quad t\in[0,T].
    \end{equation*}
    We notice that, for every $t\in[0,T]$, $Z^{i,n}_t\geq0$ a.s.  and that $\hat{\tau}^{i,n} =\inf\{t\ge 0, \,\, Z^{i,n}_t = 0\} \wedge T$. The same holds for $Z^i$ and $\hat{\tau}^i$.
    
We note that in the extreme case that
$\hat{\tau}^{i}=T$ i.e. when the level set $\{t\ge 0,\ ,\, Z^{i}_t=0\}$ is empty, the case $\hat{\tau}^{i,n}<T$ can only hold for a finite $n$. Indeed, by convergence in $\mathcal{S}^2$, up to a subsequence, $0=\underset{t\in[0,T]}{\inf}\,Z_t^{i,n}$ convergence a.s. to $\underset{t\in[0,T]}{\inf}\,Z_t^i>0$ which is impossible.  Also we may apply a similar argument to conclude that if  $\hat{\tau}^i<T$, the extreme case $\hat{\tau}^{i,n}=T$ i.e. when the level set $\{t\ge 0,\ ,\, Z^{i,n}_t=0\}$ is empty, can only hold for a finite $n$. 

So we are left with the most interesting case where both level sets 
$\{t\ge 0,\,\, Z^{i,n}_t=0\}$ and $\{t\ge 0,\,\, Z^{i,n}_t=0\}$  are not empty. For any $\eps>0$, we have 
    \begin{equation}\label{conv-prob-term}
        \P\left(\lvert \hat{\tau}^{i,n} - \hat{\tau}^i\rvert>\eps\right) = \P\left(\hat{\tau}^{i,n}-\hat{\tau}^i>\eps\right) + \P\left(\hat{\tau}^i-\hat{\tau}^{i,n}>\eps\right).
    \end{equation}
    We first show that $\underset{1\le i\le n}{\sup}\,\P\left(\hat{\tau}^{i,n}-\hat{\tau}^i>\eps\right)\to 0$ as $n\to\infty$. The event $\left\{\hat{\tau}^{i,n}-\hat{\tau}^i>\eps\right\}$
    means that $Z^{i,n}$ attains $0$ at a time which is  larger than the time $\hat{\tau}^i$, at which $Z^i$ attains the same level 0, with at least $\eps>0$. In other words,
    \begin{equation}
        \P\left(\hat{\tau}^{i,n}-\hat{\tau}^i>\eps\right) = \P\left(\inf_{0\leq t\leq\hat{\tau}^i+\eps} Z^{i,n}_t > 0, \,\, Z_{\hat{\tau}^i}^i = 0\right).
    \end{equation}

\medskip
Given $\delta >0$, let us consider the stopping time
$$
\sigma^{i,n}=\inf\{ t\ge \hat{\tau}^i, \,\, Z_t^{i,n}\le \delta\}\wedge T.
$$
If the level set $\{t\ge \tau^i,\,\, Z^{i,n}_t\le \delta\}$ is not empty, then necessarily $\sigma^{i,n}<\tau^{i,n} \,\,\as.$ Moreover,  we have $\underset{\hat{\tau}^i\le t<(\hat{\tau}^i+\eps)\wedge \sigma^{i,n}}{\inf}\, Z^{i,n}_t > \delta \,\,\as$. 
So the following holds:$$
\left\{\inf_{0\leq t\leq\hat{\tau}^i+\eps} Z^{i,n}_t > 0\right\}\subset \left\{ \underset{\hat{\tau}^i\le t < (\hat{\tau}^i+\eps)\wedge \sigma^{i,n}}{\inf}\, Z^{i,n}_t > 0\right\}=\left\{\underset{\hat{\tau}^i\le  t < (\hat{\tau}^i+\eps)\wedge \sigma^{i,n}}{\inf}\, Z^{i,n}_t> \delta\right\}.
$$ 
Therefore, 
$$\begin{aligned}
\underset{1\le i\le n} {\sup}\,\P\left(\underset{0\le t\leq\hat{\tau}^i+\eps}{\inf}\, Z^{i,n}_t > 0, \,\, Z_{\hat{\tau}^i}^i =0\right)& \le  \underset{1\le i\le n} {\sup}\,\P\left(\inf_{\hat{\tau}^i\leq t < (\hat{\tau}^i+\eps)\wedge \sigma^{i,n}} Z^{i,n}_t > \delta, \,\, Z_{\hat{\tau}^i}^i =0\right) \\ & \le \underset{1\le i\le n} {\sup}\,\P\left(\inf_{\hat{\tau}^i\leq t
        < (\hat{\tau}^i+\eps)\wedge\sigma^{i,n}} Z^{i,n}_t - Z_{\hat{\tau}^i}^i > \delta\right) \\ & \le \underset{1\le i\le n} {\sup}\,\P\left(\inf_{\hat{\tau}^i\leq t < (\hat{\tau}^i+\eps)\wedge\sigma^{i,n}} Z^{i,n}_t - \inf_{\hat{\tau}^i\leq t < (\hat{\tau}^i+\eps)\wedge\sigma^{i,n}}Z_t^i > \delta\right) \\ & \le 
        \underset{1\le i\le n} {\sup}\,\P\left ( \sup_{t\in[0,T]}\lvert Z^{i,n}_t - Z^i_t \rvert>\delta\right),
\end{aligned}
$$
 which, in view of \eqref{conv-mean-sup-Z}, entails that $\underset{1\le i\le n}{\sup}\,\P\left(\hat{\tau}^{i}-\hat{\tau}^{i,n}>\eps\right)\to 0$ as $n\to\infty$.

If the level set $\{t\ge\tau^i,\,\, Z^{i,n}_t\le \delta\}$ is empty, then $\sigma^{i,n}=T$, but this also means that  $Z^{i,n}_t> \delta \,\,\as$ for all $t\ge \tau^i$, which implies that 
$$
\left\{\inf_{0\leq t\leq\hat{\tau}^i+\eps} Z^{i,n}_t > 0\right\}\subset \left\{ \underset{\hat{\tau}^i\le t < \hat{\tau}^i+\eps}{\inf}\, Z^{i,n}_t > 0\right\}=\left\{\underset{\hat{\tau}^i\le  t < \hat{\tau}^i+\eps}{\inf}\, Z^{i,n}_t> \delta\right\}.
$$  
This in turn yields the desired result, in view of the above steps.

\medskip
    Let us now consider the second term on the right hand side of \eqref{conv-prob-term}.
Given $\rho >0$, we consider the stopping time
$$
\a^{i,n}=\inf\{ t\ge \hat{\tau}^{i,n}, \,\, Z_t^{i}\le \rho\}\wedge T.
$$
By following the steps above, we have $\underset{\hat{\tau}^{i,n}\leq t<(\hat{\tau}^{i,n}+\eps)\wedge \a^{i,n}}{\inf}\, Z^{i,n}_t > \delta$. This in turn yields
$$
\left\{ \underset{\hat{\tau}^{i,n}\leq t < (\hat{\tau}^{i,n}+\eps)\wedge \a^{i,n}}{\inf}\, Z^{i}_t > \rho\right\}=\left\{\underset{\hat{\tau}^{i,n}\leq t < (\hat{\tau}^{i,n}+\eps)\wedge \a^{i,n}}{\inf}\, Z^{i}_t> 0\right\}.
$$
Since,
$\underset{\hat{\tau}^{i,n}\leq t \le \hat{\tau}^{i,n}+\eps}{\inf}\, Z^{i}_t\le \underset{\hat{\tau}^{i,n}\leq t < (\hat{\tau}^{i,n}+\eps)\wedge\a^{i,n}}{\inf}\, Z^{i}_t$,
we have
$$
\P\left(\inf_{\hat{\tau}^{i,n}\leq t\leq\hat{\tau}^{i,n}+\eps} Z^{i}_t > 0, \,\, Z_{\hat{\tau}^{i,n}}^{i,n} =0\right)\le \P\left(\inf_{\hat{\tau}^{i,n}\leq t < (\hat{\tau}^{i,n}+\eps)\wedge \a^{i,n}} Z^{i}_t > \rho, \,\, Z_{\hat{\tau}^i}^i =0\right).
$$
Therefore,
\begin{equation*}
        \begin{aligned}
        \underset{1\le i\le n} {\sup}\,\P\left(\hat{\tau}^{i}-\hat{\tau}^{i,n}>\eps\right) 
        & \le  \underset{1\le i\le n} {\sup}\,\P\left(\inf_{\hat{\tau}^{i,n}\leq t < (\hat{\tau}^{i,n}+\eps)\wedge \a^{i,n}} Z^{i}_t > \rho, \,\, Z_{\hat{\tau}^{i,n}}^{i,n} =0\right) \\ 
        & \le \underset{1\le i\le n} {\sup}\,\P\left(\inf_{\hat{\tau}^{i,n}\leq t
        < (\hat{\tau}^{i,n}+\eps)\wedge\a^{i,n}} Z^{i}_t - Z_{\hat{\tau}^{i,n}}^{i,n} > \rho\right) \\ 
        & \le 
        \underset{1\le i\le n} {\sup}\,\P\left ( \sup_{t\in[0,T]}\lvert Z^{i,n}_t - Z^i_t \rvert>\rho\right),
\end{aligned}
    \end{equation*}
    which entails that $\underset{1\le i\le n}{\sup}\,\P\left(\hat{\tau}^{i}-\hat{\tau}^{i,n}>\eps\right)\to 0$ as $n\to\infty$, in view of \eqref{conv-mean-sup-Z}. 
    \end{proof}

\begin{corollary}\label{conv-os-prob-1} For every fixed $i\ge 1$, we have
\begin{equation}\label{BC-1}
\lim_{n\to\infty}\E[Y^i_{\hat{\tau}^{i,n}}]=\E[Y^i_{\hat{\tau}^{i}}].
\end{equation}
Moreover, up to a subsequence, it holds that 
\begin{equation}\label{BC-2}
\lim_{n\to\infty}\E\left[h(Y^{i,n}_{\hat{\tau}^{i,n}},\frac{1}{n}\sum_{j=1}^n Y^{j,n}_{\hat{\tau}^{j,n}})\ind_{\{\hat{\tau}^{i,n}<T\}}+\xi^i\ind_{\{
\hat{\tau}^{i,n}=T\}}\right]=\E\left[h(Y^i_{\hat{\tau}^i},\E[Y^i_{\hat{\tau}^i}])\ind_{\{\hat{\tau}^i<T\}}+\xi^i\ind_{\{\hat{\tau}^i=T\}}\right].
\end{equation}
\end{corollary}
\begin{proof}
   We derive \eqref{BC-1} by contradiction. Assume that $\hat{\tau}^{i,n}$ converges in probability to $\hat{\tau}^{i}$ with $|\E[Y^i_{\hat{\tau}^{i,n}}]-\E[Y^i_{\hat{\tau}^{i}}]|\ge \varepsilon >0$ for all $n$. But, then we can extract a subsequence $\hat{\tau}^{i,n_k}$ which converges to $\hat{\tau}^{i}$ a.s. Since the continuous process  $Y^i$ is in $\mathcal{S}^2$, by Dominated Convergence, we arrive at a contradiction.
   
   To derive \eqref{BC-2}, we note that since the process $Y^i$ is continuous and  $\hat{\tau}^{i,n}, \hat{\tau}^i$  are $\mathbb{F}^i$-stopping times, it holds that the sequence $(Y^{i,n},\hat{\tau}^{i,n})$ converges in probability to $(Y^{i},\hat{\tau}^{i})$. Therefore, in view of \cite{aldous81}, Corollary 16.23, $(\hat{\tau}^{i,n}, Y^{i,n}_{\hat{\tau}^{i,n}})$ converges in distribution to $(\hat{\tau}^{i}, Y^{i}_{\hat{\tau}^{i}})$. For each $i\ge 1$, let $\{\hat{\tau}^{i,n_k}\}_{k\geq1}$ be a subsequence of the sequence of stopping times $\{\hat{\tau}^{i,n}\}_{n\geq 1}$, which converges a.s. to $\hat{\tau}^{i}$. 
   We claim that for every $i\ge 1$, $\frac{1}{n_k}\sum_{j=1}^{n_k} Y^{j,n_k}_{\hat{\tau}^{j,n_k}}\overset{L^1}\to \E[Y^{i}_{\hat{\tau}^{i}}]$ as $k\to \infty$.
   Indeed, since
   $$\begin{array}{lll}
   \E\left[\left\lvert \frac{1}{n_k}\sum_{j=1}^{n_k} Y^{j,n_k}_{\hat{\tau}^{j,n_k}}- \E[Y^{i}_{\hat{\tau}^{i}}]\right\rvert\right] & \leq \E\left[\left\lvert \frac{1}{n_k}\sum_{j=1}^{n_k} (Y^{j,n_k}_{\hat{\tau}^{j,n_k}}-Y^j_{\hat{\tau}^{j,n_k}})\right\lvert\right]+\frac{1}{n_k}\sum_{j=1}^{n_k}\E\left[\left\lvert  Y^j_{\hat{\tau}^{j,n_k}}-Y^j_{\hat{\tau}^{j}}\right\lvert\right] \\ &\quad + \E\left[\left\lvert \frac{1}{n_k}\sum_{j=1}^{n_k} (Y^j_{\hat{\tau}^{j}}-\E[Y^j_{\hat{\tau}^{j}}])\right\lvert\right],
   \end{array}
  $$
  we have
  $$
  \begin{aligned}
  \E\left[\left\lvert \frac{1}{n_k}\sum_{j=1}^{n_k} (Y^{j,n_k}_{\hat{\tau}^{j,n_k}}-Y^j_{\hat{\tau}^{j,n_k}})\right\lvert\right] &\le \frac{1}{n_k}\sum_{j=1}^{n_k}\E\left[\underset{t\in [0,T]}{\sup}|Y^{j,n_k}_t-Y^j_t|\right] \\ &\le \underset{1\le j\le n_k}{\sup}\,\E\left[\underset{t\in [0,T]}{\sup}|Y^{j,n_k}_t-Y^j_t|\right]\to 0, \quad k\to\infty.
  \end{aligned}
  $$
Moreover, by Dominated Convergence, $\underset{1\le j\le n_k}{\sup}\,\E\left[\left\lvert  Y^j_{\hat{\tau}^{j,n_k}}-Y^j_{\hat{\tau}^{j}}\right\lvert\right] \to 0$ as $k\to\infty$. Thus, we have 
  $\frac{1}{n_k}\sum_{j=1}^{n_k}\E\left[\left\lvert  Y^j_{\hat{\tau}^{j,n_k}}-Y^j_{\hat{\tau}^{j}}\right\lvert\right] \le \underset{1\le j\le n_k}{\sup}\,\E\left[\left\lvert  Y^j_{\hat{\tau}^{j,n_k}}-Y^j_{\hat{\tau}^{j}}\right\lvert\right]\to 0$ as $k\to\infty.$
  
Since $\{(Y^i,\tau^i)\}_{i\ge 1}$ are i.i.d., the r.v. $Y^i_{\hat{\tau}^i},\,i=1,2,\ldots$ are i.i.d. By the strong law of large numbers and Dominated Convergence we have $\E\left[\left\lvert \frac{1}{n_k}\sum_{j=1}^{n_k} (Y^j_{\hat{\tau}^{j}}-\E[Y^j_{\hat{\tau}^{j}}])\right\lvert\right]\to 0$ as $k\to\infty$. Therefore, as $k\to \infty$, $h(Y^{i,n_k}_{\hat{\tau}^{i.n_k}},\frac{1}{n_k}\sum_{j=1}^n Y^{j,n_k}_{\hat{\tau}^{j,n_k}})\ind_{\{\hat{\tau}^{i,n_k}<T\}}+\xi^i\ind_{\{
\hat{\tau}^{i,n_k}=T\}}$ converges almost surely to $h(Y^i_{\hat{\tau}^i},\E[Y^i_{\hat{\tau}^i}])\ind_{\{\hat{\tau}^i<T\}}+\xi^i\ind_{\{\hat{\tau}^i=T\}}$. The claim \eqref{BC-2} follows by Dominated Convergence.
\end{proof}

\medskip
\begin{proof}[Proof of Theorem \ref{optimality-tau-i}]
Let $\{\hat{\tau}^{i,n_k}\}_{k\geq 1}$ be a subsequence of the sequence of stopping times $\{\hat{\tau}^{i,n}\}_{n\geq1}$, which converges a.s. to $\hat{\tau}^{i}$. In view of \eqref{BC-2} and the optimality of $\{\hat{\tau}^{i,n_k}\}_{k\geq 1}$, we have 
\begin{multline*}
Y_0^i=\underset{k\to\infty}{\lim}Y_0^{i,n_k}=\underset{k\to\infty}{\lim}\E\left[h(Y^{i,n_k}_{\hat{\tau}^{i.n_k}},\frac{1}{n_k}\sum_{j=1}^{n_k} Y^{j,n_k}_{\hat{\tau}^{j,n_k}})\ind_{\{\hat{\tau}^{i,n_k}<T\}}+\xi^i\ind_{\{
\hat{\tau}^{i,n_k}=T\}}\right] \\ =\E\left[h(Y^i_{\hat{\tau}^i},\E[Y^i_{\hat{\tau}^i}])\ind_{\{\hat{\tau}^i<T\}}+\xi^i\ind_{\{\hat{\tau}^i=T\}}\right].
\end{multline*}
\end{proof}

\section{Optimal stopping of mean-field SDEs}\label{ssec-diffusion}
Let us consider the following mean-field extension of the standard optimal stopping problem of a one-dimensional diffusion process $X$:
Find a stopping time $\tau^*$ such that 
\begin{equation}\label{Y-0-d}
\tau^*=\underset{\tau\in \mathcal{T}_0}{\arg\max}\, \E\left[h(X_{\tau},\E[X_{\tau}])\ind_{\{\tau<T\}}+\xi\ind_{\{\tau=T\}}\right],
\end{equation}
where $X$ is a diffusion process of mean-field type
\begin{equation}\label{mf-sde}
X_t=X_0+\int_0^tb(s,X_s, \E[X_{s}])ds+\int_0^t\sigma(s,X_s,\E[X_{s}])dW_s,\quad t\in [0,T],
\end{equation}
where $b$ and $\sigma$ are deterministic functions of $(t,x,y)\in [0,T]\times\R\times\R$, $X_0$ is square-integrable and independent of $W$. Here, $\mathcal{F}_t$ is the $\P$-completion of $\sigma(X_0,W_s,s\le t)$.

The OSP associated with the MF-SDE \eqref{mf-sde} is
\begin{equation}\label{Y-0-sde}
Y_0=\underset{\tau\in \mathcal{T}_0}{\esssup}\, \E\left[h(X_{\tau}, \E[X_{\tau}])\ind_{\{\tau<T\}}+\xi\ind_{\{\tau=T\}}\right].
\end{equation}

The particle system to use to solve this OSP is simply the system of i.i.d. processes $\{X^i\}_{i\geq1}$ which solve 
\begin{equation}\label{mf-sde-i}
X^i_t=X^i_0+\int_0^tb(s,X^i_s, \E[X^i_{s}])ds+\int_0^t\sigma(s,X^i_s,\E[X^i_{s}])dW^i_s,\quad t\in [0,T],
\end{equation}
 and the vector $(X^{1,n},\ldots,X^{n,n})$ of $n$ weakly interacting diffusions defined by
\begin{equation}\label{sde-i-n}
X^{i,n}_t=X_0^{i}+\int_0^tb(s,X^{i,n}_s, \frac{1}{n}\sum_{j=1}^n X^{j,n}_s)ds+\int_0^t\sigma(s,X^{i,n}_s, \frac{1}{n}\sum_{j=1}^n X^{j,n}_s)dW^i_s,\quad t\in[0,T],
\end{equation}
where $(X^1_0,W^1)=(X_0,W)$ (which implies that $X^1=X$), and $(X_0^i,W^i)$ are independent and equally distributed.

To the system \eqref{mf-sde-i} we associate the OSP
\begin{equation}\label{Y-i-sde}
Y^{i}_0=\underset{\tau\in \mathcal{T}^i_0}{\sup}\, \E\left[h(X^{i}_{\tau},\E[X^i_{\tau}])\ind_{\{\tau<T\}}+\xi^i\ind_{\{\tau=T\}}\right],\quad i\ge 1,
\end{equation}
and the associated family of optimal stopping times 
\begin{equation}\label{tau-i-sde}
    \hat{\tau}^{i}=\inf\left\{t\ge 0, \,\, Y^{i}_t=h(X^{i}_{t},\E[X^i_t])\right\}\wedge T.
    \end{equation}
Moreover, to the system \eqref{sde-i-n} we associate the family of  OSPs
\begin{equation}\label{Y-i-n-sde}
Y^{i,n}_0=\underset{\tau\in \mathcal{T}^i_0}{\esssup}\, \E\left[h(X^{i,n}_{\tau},\frac{1}{n}\sum_{j=1}^n X^{j,n}_{\tau})\ind_{\{\tau<T\}}+\xi^i\ind_{\{\tau=T\}}\right],\quad i=1,2,\ldots,n,
\end{equation}
 and the associated family of optimal stopping times 
\begin{equation}\label{tau-i-n-sde}
    \hat{\tau}^{i,n}=\inf\left\{t\ge 0, \,\, Y^{i,n}_t=\E[h(X^{i,n}_{t},\frac{1}{n}\sum_{j=1}^n X^{j,n}_{t})\, \lvert\,\mathcal{F}^i_t]\right\}\wedge T.
    \end{equation}

\begin{proposition}[Proposition 1.2 and Theorem 1.3. in \cite{meleard08}] \label{A3}
Assume $b$ and $\sigma$  are Lipschitz continuous in $(x,y)\in\R\times \R$. Then, 
\begin{itemize}
\item[(1)] Each of the $X^i$'s and $X^{i,n}$'s is in $\mathcal{S}_c^2$,
\medskip
    \item[(2)] $
\underset{n\to\infty}{\lim}\,\underset{1\le i\le n} {\sup}\,\E\left[\underset{t\in[0,T]}{\sup}|X^{i,n}_t-X^i_t|^2\right]=0$.
\end{itemize}
\end{proposition}
Based on Proposition \ref{A3}, we obtain the following 
\begin{theorem}
The hitting time
\begin{equation*}
\tau^*=\inf \{s\ge 0;\,\,  Y_s=h(X_s,\E[X_s])\}\wedge T
\end{equation*}
satisfies
\begin{equation*}
\tau^*=\underset{\tau\in \mathcal{T}_0}{\arg\max}\, \E\left[h(X_{\tau},\E[X_{\tau}])\ind_{\{\tau<T\}}+\xi\ind_{\{\tau=T\}}\right].
\end{equation*}
\end{theorem}

\subsection{Optimal stopping of the variance of a mean-field diffusion}
For $h(x,m):=(x-m)^2$ and $\xi\ge 0$, we obtain an OSP of the variance:
\begin{equation}\label{Y-0-v}
Y_0=\underset{\tau\in \mathcal{T}_0}{\sup}\,\E\left[\left(X_{\tau}-\E[X_{\tau}]\right)^2\ind_{\{\tau<T\}}+\xi\ind_{\{\tau=T\}}\right].
\end{equation}
Since $h$ is not Lipschitz continuous we cannot directly apply the above results to claim that the hitting time
\begin{equation}\label{var-tau}
\tau^*=\inf \{s\ge 0;\,\,  Y_s=(X_s-\E[X_s])^2\}\wedge T
\end{equation}
satisfies
\begin{equation}\label{var-tau-1}
\tau^*=\underset{\tau\in \mathcal{T}_0}{\arg\max}\, \,\E\left[\left(X_{\tau}-\E[X_{\tau}]\right)^2\ind_{\{\tau<T\}}+\xi\ind_{\{\tau=T\}}\right].
\end{equation}

  \ms 
 Below, we provide a proof that the hitting time $\tau^*$ defined by \eqref{var-tau} satisfies \eqref{var-tau-1}. To this end, using the notation above, we need to show similar results as those given in Proposition \ref{conv-os-prob} and Corollary \ref{conv-os-prob-1} which follow  provided that 
\begin{equation}\label{var-ex-1}
\underset{n\to\infty}{\lim}\E\left[\underset{t\in[0,T]}{\sup}\lvert Y^{1,n}_{t}- Y_{t}\lvert\right]=0 \quad \text{and}\quad \underset{n\to\infty}{\lim}\E\left[\underset{t\in[0,T]}{\sup}\lvert Z^{1,n}_{t}- Z^1_{t}\lvert\right]=0,
\end{equation}
where $Z^{1,n}$ and $Z^1$ are defined as in \eqref{Z} and
\begin{equation}\label{Y-i-n-example}
Y^{1,n}_t=\underset{\tau\in \mathcal{T}^1_t}{\esssup}\, \E\left[(X^{1,n}_{\tau}-\frac{1}{n}\sum_{j=1}^nX^{j,n}_{\tau})^2\, \lvert\,\mathcal{F}_t\right], \quad Y_t=\underset{\tau\in \mathcal{T}_t}{\esssup}\, \E\left[(X_{\tau}-\E[X_{\tau}])^2\, \lvert\,\mathcal{F}_t\right].
\end{equation}
In view of the proofs of Theorem \ref{conv-1} and Proposition \ref{conv-os-prob}, the limits \eqref{var-ex-1} hold provided that
\begin{equation}\label{var-ex-2}
\underset{n\to\infty}{\lim}\E\left[\underset{t\in[0,T]}{\sup}\lvert \E[(X^{1,n}_{t}-\frac{1}{n}\sum_{j=1}^n X^{j,n}_{t})^2\, \lvert\,\mathcal{F}_t]-(X_t-\E[X_t])^2\lvert \right]=0.
\end{equation}
Let us show \eqref{var-ex-2}. We have
\begin{equation*}
\begin{aligned}
    &\underset{t\in[0,T]}{\sup} \lvert \E[(X^{1,n}_{t}-\frac{1}{n}\sum_{j=1}^n X^{j,n}_{t})^2\, \lvert\,\mathcal{F}_t]-(X_t-\E[X_t])^2\lvert \\
    &= \underset{t\in[0,T]}{\sup} \lvert \E[((X^{1,n}_{t}-\frac{1}{n}\sum_{j=1}^n X^{j,n}_{t}) - (X_t-\E[X_t]))  ((X^{1,n}_{t}-\frac{1}{n}\sum_{j=1}^n X^{j,n}_{t}) + (X_t-\E[X_t])) \lvert\,\mathcal{F}_t]\lvert\\
    &\leq \underset{t\in[0,T]}{\sup}  \E[\underset{s\in[0,T]}{\sup}\lvert(X^{1,n}_{s}-\frac{1}{n}\sum_{j=1}^n X^{j,n}_{s}) - (X_s-\E[X_s])\rvert\lvert(X^{1,n}_{s}-\frac{1}{n}\sum_{j=1}^n X^{j,n}_{s}) + (X_s-\E[X_s])\rvert \lvert\,\mathcal{F}_t]\\
    &\qquad\leq \underset{t\in[0,T]}{\sup}  \left(\E[\underset{s\in[0,T]}{\sup}\lvert(X^{1,n}_{s}-\frac{1}{n}\sum_{j=1}^n X^{j,n}_{s}) - (X_s-\E[X_s])\rvert^2\lvert\,\mathcal{F}_t]\right)^{\frac{1}{2}}\\
    &\qquad\qquad \underset{t\in[0,T]}{\sup}\left(\E[\underset{s\in[0,T]}{\sup}\lvert(X^{1,n}_{s}-\frac{1}{n}\sum_{j=1}^n X^{j,n}_{s}) + (X_s-\E[X_s])\rvert^2\lvert\,\mathcal{F}_t]\right)^{\frac{1}{2}},
\end{aligned}
\end{equation*}
where in the last inequality we used the Cauchy-Schwarz inequality for the conditional expectation. Then, again by Doob's inequalities, we obtain
\begin{equation*}
\begin{aligned}
     &\E\left[\underset{t\in[0,T]}{\sup}\lvert \E[(X^{1,n}_{t}-\frac{1}{n}\sum_{j=1}^n X^{j,n}_{t})^2\, \lvert\,\mathcal{F}_t]-(X_t-\E[X_t])^2\lvert\right]\\
     &\leq 4 \left(\E[\underset{t\in[0,T]}{\sup}\lvert(X^{1,n}_{t}-\frac{1}{n}\sum_{j=1}^n X^{j,n}_{t}) - (X_t-\E[X_t])\rvert^2]\right)^\frac{1}{2}\\
     &\quad\quad \left(\E[\underset{t\in[0,T]}{\sup}\lvert(X^{1,n}_{t}-\frac{1}{n}\sum_{j=1}^n X^{j,n}_{t}) + (X_t-\E[X_t])\rvert^2]\right)^\frac{1}{2}\\ & \qquad\qquad\qquad\qquad\leq C \left(\E[\underset{t\in[0,T]}{\sup}\lvert(X^{1,n}_{t}-\frac{1}{n}\sum_{j=1}^n X^{j,n}_{t}) - (X_t-\E[X_t])\rvert^2]\right)^{1/2},
\end{aligned}
\end{equation*}
where in the last inequality we have used the fact that the term $\E[\underset{t\in[0,T]}{\sup}\lvert(X^{1,n}_{t}-\frac{1}{n}\sum_{j=1}^n X^{j,n}_{t})+(X_t-\E[X_t])\rvert^2]$ is bounded by a constant $C$ which only depends on the  $\mathcal{S}^2$-norm of $X$, due to (2) in Proposition \ref{A3} and the exchangeability of the sequence $\{X^{j,n}\}^n_{j=1}$. Furthermore, we have
\begin{equation*}
\begin{aligned}
     &\E[\underset{t\in[0,T]}{\sup}\lvert(X^{1,n}_{t}-\frac{1}{n}\sum_{j=1}^n X^{j,n}_{t}) - (X_t-\E[X_t])\rvert^2] \\ &\qquad \le 2 \E[\underset{t\in[0,T]}{\sup}\lvert X^{1,n}_{t}-X_t\lvert^2]  +2\E[\underset{t\in[0,T]}{\sup}\lvert\frac{1}{n}\sum_{j=1}^n X^{j,n}_{t}-\E[X_t])\rvert^2]\\ & \qquad \le 2 \E[\underset{t\in[0,T]}{\sup}\lvert X^{1,n}_{t}-X_t\lvert^2] + 4\E[\underset{t\in[0,T]}{\sup}\lvert\frac{1}{n}\sum_{j=1}^n (X^{j,n}_{t}-X^{j}_{t})\rvert^2] \\ & \qquad \qquad +4\E[\underset{t\in[0,T]}{\sup}\lvert\frac{1}{n}\sum_{j=1}^n (X^{j}_{t}-\E[X_t])\rvert^2].
\end{aligned}
\end{equation*}
Again, by Proposition \ref{A3} (ii) and the exchangeability of the processes $\{X^{j,n}-X^j\}^n_{j=1}$,  the first two terms in the last inequality  go 0 as $n$ goes to infinity. Now,  since the processes $\{X^j\}_{j\geq1}$ are i.i.d. $C([0,T];\R)$-valued random variables with finite second moments (since they are in $\mathcal{S}^2$), by the strong law of large numbers for Banach-valued r.v. (see Theorem 4.1.1 in \cite{padgett-taylor06}) and Dominated Convergence, it holds that
\begin{equation*}
    \underset{n\to\infty}{\lim}\E[\underset{t\in[0,T]}{\sup}\lvert\frac{1}{n}\sum_{j=1}^n (X^{j}_{t}-\E[X_t])\rvert^2]=0.
\end{equation*}
This finishes the proof of \eqref{var-ex-2}.
\medskip

\section{Optimal stopping of the variance of a Markov diffusion}\label{osp-markov}

Let $X$ be the one dimensional (time homogeneous) Markov diffusion process satisfying the SDE
\begin{equation}\label{markov-d}
dX_t=b(X_t)dt+\sigma(X_t)dW_t,\quad t\ge 0; \quad X_0=x    
\end{equation}
where $b$ and $\sigma$ are deterministic function which are Lipschitz continuous and of linear growth. We will denote the unique strong solution of \eqref{markov-d} by $X_t=X_t^x, \,t\ge 0$. We also use the 'abuse of'  notation $\E_x[f(X_t)]=\E[f(X_t^x)]$ (see e.g. \cite{oksendal13}, Eq. (7.1.7)).

In this section we provide the main ingredients of the limit approach of Section \ref{ssec-diffusion}, which lead to the proof of optimality of the 'pre-committed' hitting time $\tau^*(x)$ defined by 
\begin{equation}\label{var-tau-mark}
\tau^*(x)=\inf \{s\ge 0;\,\,  Y_s(x)=(X^x_s-\E[X^x_s])^2\}\wedge T,
\end{equation}
where
\begin{equation}\label{Y-t-v-mark}
Y_t(x)=\underset{\tau\in \mathcal{T}_t}{\esssup}\,\E\left[\left(X^x_{\tau}-\E[X^x_{\tau}]\right)^2\ind_{\{\tau<T\}}+\xi\ind_{\{\tau=T\}}\,|\,\Fc_t\right],
\end{equation}
satisfies
\begin{equation}\label{var-tau-1-mark}
\tau^*(x)=\underset{\tau\in \mathcal{T}_0}{\arg\max}\, \,\E_x\left[\left(X^x_{\tau}-\E[X^x_{\tau}]\right)^2\ind_{\{\tau<T\}}+\xi\ind_{\{\tau=T\}}\right].
\end{equation}

 Recall that by the Burkholder-Davis-Gundy (BDG) and Gronwall's inequalities, for each $x\in\R$,  $X^x\in\Sc_c^2$. More precisely, the following estimate holds for some $C>0$ depending only on the linear growth constants of the coefficients and the time horizon $T$.
\begin{equation}\label{lg}
\E[\underset{t\in[0,T]}{\sup}\left\lvert X^x_t\right\rvert^2]\le C(1+|x|^2).
\end{equation}

We consider the system of independent processes $\{X^i\}_{i\geq1}$ which are strong solutions of 
\begin{equation}\label{sde-i}
X^i_t=x_i+\int_0^tb(X^i_s)ds+\int_0^t\sigma(X^i_s)dW^i_s,\quad t\in [0,T],
\end{equation}
where each $x_i\in \R$ with $x_1=x$, $W^1=W$. In particular, $X^1=X^x$. Moreover, $W^i$'s are independent copies of the Brownian motion $W$.  

Again, by the BDG and Gronwall's inequalities, for each $x_n\in\R$,  we have
\begin{equation}\label{lg-n}
\E[\underset{t\in[0,T]}{\sup}\left\lvert X^n_t\right\rvert^2]\le C(1+|x_n|^2).
\end{equation}

The sequence of Snell envelops which approximates $Y(x)$ is 
\begin{equation}\label{Y-t-v-mark-i-n}
Y^{1,n}_t=\underset{\tau\in \mathcal{T}_t}{\esssup}\,\E\left[\left(X^1_{\tau}-\frac{1}{n}\sum_{j=1}^nX^j_{\tau}\right)^2\ind_{\{\tau<T\}}+\xi\ind_{\{\tau=T\}}\,|\,\Fc^1_t\right].
\end{equation}
The associated sequence of optimal stopping times is
\begin{equation}\label{markov-tau-i-n-sde}
    \hat{\tau}^{1,n}=\inf\left\{t\ge 0, \,\, Y^{1,n}_t=\E[(X^{1}_{t}-\frac{1}{n}\sum_{j=1}^n X^{j}_{t})^2\, \lvert\,\mathcal{F}^1_t]\right\}\wedge T.
    \end{equation}

By applying the proofs of Theorem \ref{conv-1} and Proposition \ref{conv-os-prob}, the convergence in probability of the sequence of optimal stopping times $\{\hat{\tau}^{1,n}\}_{n\ge 1}$ to optimal stopping time $\tau^*(x)$ holds provided that

\begin{equation}\label{marokv-conv-1}
\underset{n\to\infty}{\lim}\, \E\left[\underset{t\in[0,T]}{\sup}\left\lvert \frac{1}{n}\sum_{i=1}X^i_t-\E[X^x_t]\right\rvert^2\right]=0.
\end{equation}
We have
\begin{proposition}\label{markov-conv} By choosing the sequence $\{x_i\}_{i\ge 1}$ such that
\begin{equation}\label{marokv-conv-2}
\underset{i\ge 1}{\sup}\,|x_i|<\infty, \quad \underset{n\to\infty}{\lim}\,\frac{1}{n}\sum_{i=1}^n \lvert x_i-x\rvert=0,
\end{equation}
the limit \eqref{marokv-conv-1} holds.
\end{proposition}
\begin{proof}
The proof of \eqref{marokv-conv-1} follows by applying the law of large number for uniformly integrable independent r.v. we recall in  Lemma \ref{LLN-X} in the appendix.

The processes 
$$
X^i_t-\E[X^x_t]=x_i-x+\int_0^t (b(X^i_s)-E[b(X^x_s)])ds+\int_0^t\sigma(X^i_s)dW^is,\quad i\ge 1,
$$
being independent, the random variables 
$$
\overline{X}_n:=\underset{t\in[0,T]}{\sup}\lvert X^n_t-\E[X^x_t]\rvert, \quad n\ge 0,
$$
are independent. Moreover, by the Cauchy-Schwarz inequality we have
$$
\E\left[\underset{t\in[0,T]}{\sup}\left\lvert \frac{1}{n}\sum_{i=1}X^i_t-\E[X^x_t]\right\rvert^2\right]\le \frac{1}{n}\E[\lvert \overline{X}_n \rvert^2].
$$
But, in view of \eqref{lg-n} and the first condition of \eqref{marokv-conv-2}, we have
$$
\underset{n\ge 1}{\sup}\,\E[\lvert \overline{X}_n \rvert^2]\le C(1+\underset{n\ge 1}{\sup}\,|x_n|^2)<\infty.
$$
Hence, the sequence $\{\overline{X}_n\}_{n\ge 1}$ satisfies the conditions of Lemma \ref{LLN-X} which, in view of the second condition of \eqref{marokv-conv-2}, yields 
$$
\underset{n\to\infty}{\lim}\,\frac{1}{n}\sum_{j=1}^n \E[|\overline{X}_n|^2]=0,
$$
which implies \eqref{marokv-conv-1}.
\end{proof}

\newpage

\section{Acknowledgments}
We would like to thank the anonymous referee for his insightful remarks and suggestions. We also extend our thanks to  S. Hamad{\`e}ne and R. Dumitrescu  for their comments on an early version of the paper that led to correct an error in the proof of Theorem \ref{conv-1}.

\bibliographystyle{plain} 
\bibliography{refs} 

\section{Appendix}
In this section we give a proof of the following lemma we implicitly used in the proof of equation \eqref{1-app}. 
\begin{lemma}\label{stop-app} Let $X$ be an $\F$-adapted  continuous process. Then 
$$
\underset{t\in [0,T]}{\sup}\, X_t=\underset{t\in [0,T]}{\esssup}\, X_{t}=\underset{\tau\in\Tc_0}{\esssup}\,X_{\t}\,\,\, \as.
$$ 
\end{lemma}

\begin{proof}
Since $X$ is (right)-continuous and adapted, we have $\underset{t\in[0,T]}{\sup}\,X_t=\underset{t\in[0,T]\cap \Q}{\sup}\, X_t\,\,\as$, where $\Q$ denotes the set of rational numbers. Therefore, $\underset{t\in[0,T]}{\sup}\,X_t$ is a random variable. By the uniqueness of the essential supremum, we have $\underset{t\in [0,T]}{\sup}\, X_t=\underset{t\in [0,T]}{\esssup}\, X_{t}\,\, \as$. Furthermore, 
for every $\tau\in\Tc_0$, the sequence $(\tau_n)_n$ of $\F$-stopping times defined by
$\tau_n(\omega)= \frac{k}{2^n}T$ if $\frac{(k-1)}{2^n}T\le \tau(\omega)< \frac{k}{2^n}T,\,\, k=1,2,\ldots, 2^n$, $\,\,\,=T$ if $\tau(\omega)=T$, has discrete range and $\tau_n\downarrow \tau$. By right continuity of $X$, we have $X_{\tau}=\lim_{n\to\infty}X_{\tau_n} \,\,\as$. Thus,
\begin{equation*}
\begin{aligned}
X_{\tau}=\lim_{n\to\infty}X_{\tau_n} & =\underset{n\to\infty}{\lim}\sum_{k=1}^{2^n}X_{\frac{k}{2^n}T}\ind_{\{\frac{(k-1)}{2^n}T\le \tau < \frac{k}{2^n}T\}}+X_T\ind_{\{\tau=T\}} \\ 
& \le \underset{t\in [0,T]}{\esssup}\, X_{t} \,\left(\underset{n\to\infty}{\lim}\sum_{k=1}^{2^n}\ind_{\{\frac{(k-1)}{2^n}T\le \tau < \frac{k}{2^n}T\}}+\ind_{\{\tau=T\}}\right) \\ & =\underset{t\in [0,T]}{\esssup}\, X_{t}  \,\ind_{\{0\le \tau\le T\}}=\underset{t\in [0,T]}{\esssup}\, X_{t} \,\,\,\as.
\end{aligned}
\end{equation*}
Hence, $\underset{\tau\in \Tc_0}{\esssup}\, X_{\tau}\le \underset{t\in [0,T]}{\sup}\, X_t\,\,\as$. The reverse inequality follows from the fact that each $t\in [0,T]$ is in $\Tc_0$ and so 
$X_t\le \underset{\tau\in \Tc_0}{\esssup}\, X_{\tau}\,\,\as$. 
\end{proof}
\begin{lemma}[A law of large numbers]\label{LLN-X}
Suppose that $\{X_n\}_{n\ge 1}$ are independent random variables such that 
$$
\underset{n\ge 1}{\sup}\, \E[|X_n|^2]<\infty.
$$
Then
$$
\frac{1}{n}\sum_{j=1}^n \E[|X_j|^2] \longrightarrow 0 \quad \text{as}\quad n\to \infty.
$$
\end{lemma}

\begin{proof} For every $a>0$, we have
$$\begin{aligned}
\frac{1}{n}\sum_{j=1}^n \E[|X_j|^2] & =\frac{1}{n}\sum_{j=1}^n \E[|X_j|^2\ind_{\{|X_j|\le a\}}]+\frac{1}{n}\sum_{j=1}^n \E[|X_j|^2\ind_{\{|X_j|>a\}}]\\ & \le \frac{a}{n}\sum_{j=1}^n \E[|X_j|\ind_{\{|X_j|\le a\}}]+ \underset{n\ge 1}{\sup}\,\E[|X_n|^2\ind_{\{|X_n|> a\}}].
\end{aligned}
$$
We have 
$$
\frac{S_n}{n}:=\frac{1}{n}\sum_{j=1}^n |X_j|\ind_{\{|X_j|\le a\}} \overset{\as}{\longrightarrow} 0 \quad \text{as}\quad n\to \infty.
$$
Indeed, by independence of the $X_n$'s, for every $\eps>0$, 
$$
\P(S_{m^2}>m^2\eps)\le \frac{1}{m^4 \eps^2}\sum_{j=1}^{m^2}\E[|X_j|^2\ind_{\{|X_j|\le a\}}]\le \frac{a^2}{\eps^2 m^2},
$$
and so, by Borel-Cantelli lemma, $S_{m^2}/m^2\to 0$ almost surely as $m\to\infty$. Further, consider the largest deviation from $S_{m^2}$ that can occur between $m^2$ and $(m+1)^2$ defined by 
$$
L_m:=\underset{m^2\le n<(m+1)^2}{\max}\, (S_n-S_{m^2}). 
$$
We have
$$
\E[L_m^2]\le \sum_{n=m^2}^{(m+1)^2-1}\E[(S_n-S_{m^2})^2]\le 4a^2m^2.
$$
Therefore, 
$$
\P(L_m>m^2\eps)\le \frac{4a^2}{\eps^2 m^2}
$$
and by Borel-Cantelli lemma it holds that $L_{m^2}/m^2\to 0$ almost surely as $m\to\infty$. Hence, for $n$ between $m^2$ and $(m+1)^2$, we have
$$
\frac{S_n}{n}\le \frac{S_{m^2}+L_m}{n}\le \frac{S_{m^2}+L_m}{m^2} \overset{\as}{\longrightarrow}0 \quad \text{as}\quad n\to\infty.
$$
We have shown that $S_n/n\to 0$ almost surely, as $n\to\infty$, and  by dominated convergence, $\frac{1}{n}\E[S_n]\to 0$ as $n\to \infty$.

Hence, for every $a>0$,
$$
\underset{n\to\infty}{\Lim}\frac{1}{n}\sum_{j=1}^n \E[|X_j|^2]\le \underset{n\ge 1}{\sup}\,\E[|X_n|^2\ind_{\{|X_n|> a\}}].
$$
Since $a$ is arbitrarily chosen, by the uniform integrability of the sequence $\{X_n\}_{n\ge 1}$, we obtain
$$
\underset{n\to\infty}{\Lim}\frac{1}{n}\sum_{j=1}^n \E[|X_j|^2]\le \underset{a\to\infty}{\Lim}\, \underset{n\ge 1}{\sup}\,\E[|X_n|^2\ind_{\{|X_n|> a\}}]=0.
$$

\end{proof}

\end{document}